\documentclass[english,review]{elsarticle}
\usepackage[T1]{fontenc}
\usepackage[latin9]{inputenc}
\usepackage{float}
\usepackage{mathtools}
\usepackage{amsmath}
\usepackage{amsthm}
\usepackage{amssymb}
\usepackage{graphicx}

\makeatletter

\floatstyle{ruled}
\newfloat{algorithm}{tbp}{loa}
\providecommand{\algorithmname}{Algorithm}
\floatname{algorithm}{\protect\algorithmname}

\theoremstyle{plain}
\newtheorem{thm}{\protect\theoremname}
\theoremstyle{definition}
\newtheorem{defn}[thm]{\protect\definitionname}
\theoremstyle{plain}
\newtheorem{lem}[thm]{\protect\lemmaname}
\theoremstyle{plain}
\newtheorem{prop}[thm]{\protect\propositionname}
\ifx\proof\undefined
\newenvironment{proof}[1][\protect\proofname]{\par
	\normalfont\topsep6\p@\@plus6\p@\relax
	\trivlist
	\itemindent\parindent
	\item[\hskip\labelsep\scshape #1]\ignorespaces
}{%
	\endtrivlist\@endpefalse
}
\providecommand{\proofname}{Proof}
\fi
\theoremstyle{remark}
\newtheorem{rem}[thm]{\protect\remarkname}

\usepackage{algorithm}
\usepackage{algpseudocode}
\usepackage{lineno}

\@ifundefined{showcaptionsetup}{}{%
 \PassOptionsToPackage{caption=false}{subfig}}
\usepackage{subfig}
\makeatother

\usepackage{babel}
\providecommand{\definitionname}{Definition}
\providecommand{\lemmaname}{Lemma}
\providecommand{\propositionname}{Proposition}
\providecommand{\remarkname}{Remark}
\providecommand{\theoremname}{Theorem}

\begin{document}

\begin{frontmatter}{}

\title{Robust Discontinuity Indicators for High-Order \\
Reconstruction of Piecewise Smooth Functions}

\author[sbu]{Yipeng Li}

\author[sbu]{Qiao Chen}

\author[sbu]{Xiangmin Jiao\corref{cor1}}

\ead{xiangmin.jiao@stonybrook.edu}

\cortext[cor1]{Corresponding author}

\address[sbu]{Dept. of Applied Math. \& Stat. and Institute for Advanced Computational
Science, Stony Brook University, Stony Brook, NY 11794, USA.}
\begin{abstract}
In many applications, piecewise continuous functions are commonly
interpolated over meshes. However, accurate high-order manipulations
of such functions can be challenging due to potential spurious oscillations
known as the Gibbs phenomena. To address this challenge, we propose
a novel approach, \emph{Robust Discontinuity Indicators} (\emph{RDI}),
which can efficiently and reliably detect both $C^{0}$ and $C^{1}$
discontinuities for node-based and cell-averaged values. We present
a detailed analysis focusing on its derivation and the dual-thresholding
strategy. A key advantage of RDI is its ability to handle potential
inaccuracies associated with detecting discontinuities on non-uniform
meshes, thanks to its innovative discontinuity indicators. We also
extend the applicability of RDI to handle general surfaces with boundaries,
features, and ridge points, thereby enhancing its versatility and
usefulness in various scenarios. To demonstrate the robustness of
RDI, we conduct a series of experiments on non-uniform meshes and
general surfaces, and compare its performance with some alternative
methods. By addressing the challenges posed by the Gibbs phenomena
and providing reliable detection of discontinuities, RDI opens up
possibilities for improved approximation and analysis of piecewise
continuous functions, such as in data remap.
\end{abstract}

\end{frontmatter}{}


\section{Introduction}

The robust and efficient detection of discontinuities in piecewise
continuous functions plays a crucial role in a wide range of computational
applications, such as computer-aided design (CAD), computer graphics,
and computational fluid dynamics (CFD) \citep{hesthaven2008,botsch2010,leveque2002}.
If these discontinuities are not accurately identified and resolved,
they can lead to non-physical oscillations, a phenomenon known as
the \emph{Gibbs phenomenon} \citep{gottlieb1997}.

Several high-order numerical methods, such as the Discontinuous Galerkin
(DG) methods \citep{hesthaven2008} and Weighted Essentially Non-Oscillatory
(WENO) schemes \citep{liu1994weighted}, have been developed to tackle
these challenges. These methods excel at handling sharp transitions
and maintaining high-order accuracy when solving hyperbolic partial
differential equations (PDEs). However, they often encounter difficulties
in robustly detecting discontinuities.

In this paper, we present a novel method named \emph{Robust Discontinuity
Indicators} (\emph{RDI}), designed specifically to overcome the limitations
of existing methods by providing robust and efficient detection of
discontinuities in function values or derivatives, commonly referred
to as $C^{0}$ and $C^{1}$ discontinuities, respectively. RDI is
applicable for both node-based and cell-averaged values. We developed
RDI by extending the discontinuity indicators in our previous work
\citep{li2020wls}, where the analysis was more heuristic and high-level.
We present a more rigorous analysis, including the derivation of the
thresholds, and enhance the robustness of the indicators for detecting
discontinuities on non-uniform meshes.

RDI is particularly notable for its ability to support general surfaces
with complex features, such as boundaries, sharp ridges, and corners,
which are common in practical applications \citep{botsch2010}. This
versatility of RDI opens up new possibilities for its application
in computational physics and beyond. By accurately detecting and characterizing
discontinuities, RDI offers improved approximation and analysis of
piecewise continuous functions, enabling more reliable simulations,
data remapping, and other computational tasks. The development of
RDI fills a critical gap in the field, providing a powerful tool for
handling discontinuities and mitigating the adverse effects of the
Gibbs phenomena.

The remainder of the paper is organized as follows: Section~\ref{sec:Background-and-Preliminaries}
provides a brief review of weighted least squares approximations and
the Gibbs phenomena. Section~\ref{sec:Element-Based-Indicators}
analyzes overshoots and undershoots near discontinuities, leading
to the derivation of element-based indicators. Section~\ref{sec:Node-based-Markers}
introduces node-based markers using a dual-thresholding strategy and
discusses their applicability to non-uniform meshes and surfaces with
boundaries and sharp features. Section~\ref{sec:Numerical-Results}
presents several numerical experiments to demonstrate the robustness
of RDI and compares it with alternative approaches. Finally, Section~\ref{sec:Conclusions}
concludes the paper with a summary of the findings and possible future
research directions.

\section{\label{sec:Background-and-Preliminaries}Background and Preliminaries}

Weighted least square (WLS) is a powerful method used in data fitting,
remapping, reconstruction, etc. However, degree-$p$ WLS with $p\geq2$,
can introduce oscillations near discontinuities \citep{li2020wls}.
Interestingly, these oscillations can also be useful in detecting
discontinuities. In this section, we provide an overview of WLS and
discuss the phenomenon of oscillations, commonly referred to as the
Gibbs phenomenon.

\subsection{Related work\label{subsec:Related-work}}

Discontinuities often lead to undesired overshoots or undershoots
in numerical methods. Specifically, $C^{0}$ discontinuities, also
known as edges \citep{tadmor2007filters}, jump discontinuities \citep{cates2007detecting,hewitt1979gibbs}
or faults \citep{bozzini2013detection,de2002vertical}, tend to cause
$O(1)$ oscillations that persist with mesh refinement \citep{li2020wls}.
This phenomenon is referred to the Gibbs or Gibbs-Wilbraham phenomenon
in the context of the Fourier transform \citep{hewitt1979gibbs}.
While $C^{0}$ discontinuities have been extensively studied, less
attention has been given to $C^{1}$ discontinuities, which are discontinuities
in the gradients of functions. Nonetheless, $C^{1}$ discontinuities
are still crucial in various applications. For instance, they often
represent junctions or intersections of components, such as features
or ridge points. The $C^{1}$ discontinuities may cause mild oscillations
in numerical methods \citep{moura2022gradient}.

Detecting these discontinuities is essential for the effective resolution
of the Gibbs phenomenon in numerical methods. It also has significant
implications for various other applications, such as mesh reconstruction
\citep{li2019compact} and image processing \citep{shrivakshan2012comparison}.
Several methods have been developed for edge detection, such as the
minmod edge detection (MED) \citep{archibald2005polynomial,saxena2009high}
based on the polynomial annihilation technique \citep{archibald2008determining},
zero-crossing in the second derivatives \citep{marr1980theory,marr1976early,marr1979bandpass}
and the singularity detection and processing with wavelets (SDPW)
\citep{mallat1992singularity} based on wavelet transform. Among these
methods, MED focuses on a 2D unstructured mesh rather than pure image
processing. It approximates the jump function to detect the $C^{0}$
discontinuities and utilizes MED on an approximation of the gradient
to detect the $C^{1}$ discontinuities. However, the approximation
of gradient near the $C^{1}$ discontinuities may diffuse and result
in an oscillated smooth function, leading to some false negatives
in the detection. In Section~\ref{subsec:Comparison-with-minmod},
we will compare our method with MED.

\subsection{\label{subsec:Weighted-least-squares-approxima}Weighted-least-squares
approximations}

In numerical analysis, \textit{approximation} refers to using a simple
function to simulate a more complex function or a function without
an analytic formula. In practice, these functions might only have
values at a discrete set of points, which necessitates the use of
\emph{interpolation} to obtain an approximation from these sample
points. Degree-$p$ weighted least square is a powerful method that
employs polynomials of degree up to $p$ to approximate a function
sampled at discrete points. Taking degree-$2$ WLS as an example,
we can expand the $n$-dimensional Taylor series of a function $f$
at a point $\boldsymbol{x}=\boldsymbol{x}_{0}+\boldsymbol{h}$ using
the first three terms,
\begin{equation}
f(\boldsymbol{x}_{0}+\boldsymbol{h})=f(\boldsymbol{x}_{0})+\boldsymbol{h}^{T}\boldsymbol{\nabla}f(\boldsymbol{x}_{0})+\frac{1}{2}\boldsymbol{h}^{T}\boldsymbol{H}f(\boldsymbol{x}_{0})\boldsymbol{h}+\mathcal{O}(\Vert\boldsymbol{h}\Vert^{3}),\label{eq:taylor-series}
\end{equation}
where $\boldsymbol{\nabla}$ and $\boldsymbol{H}$ denote the gradient
and Hessian operators, respectively. 

By considering the approximation $f_{Q}(\boldsymbol{x})=f(\boldsymbol{x}_{0})+\boldsymbol{h}^{T}\tilde{\boldsymbol{g}}+\frac{1}{2}\boldsymbol{h}^{T}\tilde{\boldsymbol{H}}\boldsymbol{h}$,
where $\tilde{\boldsymbol{g}}$ and $\tilde{\boldsymbol{H}}$ denote
the approximations to $\boldsymbol{\nabla}f(\boldsymbol{x}_{0})$
and $\boldsymbol{H}f(\boldsymbol{x}_{0})$, respectively, we can express
each component of $\tilde{\boldsymbol{g}}$ and $\tilde{\boldsymbol{H}}$
as a linear combination of $f_{i}=f(\boldsymbol{x}_{i})$, i.e., 
\[
\tilde{g}_{j}=\sum_{k}c_{ik}f_{k}\qquad\text{and}\qquad\tilde{h}_{ij}=\sum_{k}d_{ij,k}f_{k}.
\]
Here, $c_{ik}$ and $d_{ij,k}$ are coefficients that depend on the
weights used in the WLS approximation. These coefficients have magnitudes
of $\Theta(1/h)$ and $\Theta(1/h^{2})$, respectively, and are independent
of the function $f$. When $f$ is at least $C^{2}$ continuous, the
quadratic WLS approximation $f_{Q}(\boldsymbol{x})$ is at least second-order
accurate. To determine the coefficients, we construct a linear system
\[
\boldsymbol{A}\boldsymbol{c}=\boldsymbol{f},
\]
where $\boldsymbol{A}$ is the generalized Vandermonde matrix and
$\boldsymbol{c}$ contains the coefficients to be determined. By including
an adequate number of sampling points in the stencil, we ensure the
well conditioning of the generalized Vandermonde matrix after scaling.
We assign weights to each equation to emphasize the importance of
the corresponding sample point while minimizing the least square of
the residual $\Vert\boldsymbol{W}(\boldsymbol{A}\boldsymbol{c}-\boldsymbol{f})\Vert_{2}$.
Different weight choices, such as distance-inverse weights \citep{JZ08CCF}
or radial basis function (RBF) weights \citep{buhmann2003radial}
based on distances, have been used in WLS, each suited for different
scenarios. Notably, the WLS-WENO weights \citep{hu1999weighted,shi2002technique,xu2011point,zhang2009third,liuzhang2013robust}
are based on function values and aim to resolve discontinuities.
\begin{defn}
A degree-$p$ WLS is \emph{stable} if the scaled Vandermonde matrix
has a bounded condition number independently of $h$.
\end{defn}
Here, $h$ represents the mesh size. Stability is critical in ensuring
the accuracy of the approximation under mesh refinement. Note that
interpolation is a special case of WLS where the Vandermonde matrix
is square. 
\begin{lem}
\label{lem:wls_accuracy} Given a $C^{p}$ function $f$, of which
the $(p+1)$st derivative is bounded, a stable degree-$p$ WLS $f_{p}$
is $(p+1)$st order accurate, i.e., 
\[
\left\vert f_{p}-f\right\vert =\mathcal{O}(h^{p+1}).
\]
\end{lem}
We omit the proof since it is similar to the arguments in \citep{JZ08CCF}.
Lemma~\ref{lem:wls_accuracy} will be used later in Sections~\ref{sec:Element-Based-Indicators}
and \ref{sec:Node-based-Markers} to derive our indicators in RDI.

\subsection{Weighted averaging of local fittings}

In this subsection, we review the concept of \emph{Weighted Averaging
of Local Fittings} (\emph{WALF}) \citep{Jiao2011RHO}, a technique
for reconstructing functions on discrete surfaces $\Omega\subset\mathbb{R}^{3}$.
Given a point $\boldsymbol{p}$ located within an element with vertices
$\{\boldsymbol{x}_{i}\}_{i=1}^{m}$, let $\boldsymbol{\xi}$ denote
the natural coordinates of $\boldsymbol{p}$ such that $\boldsymbol{p}=\sum_{i=1}^{m}\xi_{i}\boldsymbol{x}_{i}$.
For each vertices $\boldsymbol{x}_{i}$, we obtain an approximation
$g_{i}(\boldsymbol{p})$ by employing WLS fitting on the projected
function $f_{i}(\boldsymbol{u})\coloneqq f(\boldsymbol{x})$ defined
on $\Omega$, where $\boldsymbol{u}=\phi_{i}(\boldsymbol{x})\in\mathbb{R}^{2}$
represents the projection of $\boldsymbol{x}$ onto the normal plane
at $\boldsymbol{x}_{i}$, with $\boldsymbol{x}_{i}$ serving as the
origin of the local coordinate system. The WALF reconstruction for
$\boldsymbol{p}$ is then given by

\begin{equation}
g_{\textrm{WALF}}(\boldsymbol{p})=\sum_{i=1}^{m}\xi_{i}g_{i}(\boldsymbol{p}).\label{eq:WALF}
\end{equation}
For further details on WALF, we refer readers to \citep{Jiao2011RHO}.

To construct the WLS approximation $g_{i}$ for each sampled point
$\boldsymbol{x}_{i}$, we use the Buhmann weights in \citep{li2020wls,buhmann2001new}.
These weights are solely determined by the distance and are independent
of the function values. By precomputing $g_{i}$ using this approach,
we can reuse $g_{i}$ for all the points in the proximity of $\boldsymbol{x}_{i}$.
Compared to pure WLS and the \emph{Continuous Moving Frames} (CMF)
method \citep{Jiao2011RHO,RayWanJia12}, where a separate WLS approximation
is constructed for each target point, WALF offers computational advantages
when the number of sampled points is smaller than the number of target
points. However, it is important to note that WALF is more susceptible
to oscillations near discontinuities than CMF, as demonstrated in
\citep{Jiao2011RHO}. This oscillatory property of WALF turns out
to be useful in detecting discontinuities, as we will discuss in Section~\ref{subsec:Precomputing-operator-for}.

\section{\label{sec:Element-Based-Indicators}Element-Based Indicators}

Our technique shares similarities with that of Li et al. \citep{li2020wls}.
However, the discussions in \citep{li2020wls} were relatively high
level. We present a more rigorous analysis, including the derivation
of the thresholds.

\subsection{Asymptotic bounds in smooth regions}

Consider a function $f$ defined over a surface mesh $\Omega$ embedded
in $\mathbb{R}^{3}$. To derive discontinuity indicators, we start
with establishing asymptotic bounds that quantify the difference between
stable quadratic polynomial fitting and stable linear interpolation.
These bounds provide insights into the accuracy of our approach and
its ability to capture smooth variations in the underlying function.
\begin{prop}
\label{prop:smooth-error}In $C^{2}$ regions of a piecewise smooth
function over a surface mesh $\Omega$, the difference between stable
quadratic WLS and stable linear interpolation is $\mathcal{O}(h^{2})$.
\end{prop}
\begin{proof}
Without loss of generality, let us first assume $\Omega$ is composed
of triangles. Let $f_{L}(\boldsymbol{x}):\Omega\rightarrow\mathbb{R}$
and $f_{Q}(\boldsymbol{x}):\Omega\rightarrow\mathbb{R}$ denote the
linear interpolation and quadratic WLS of a $C^{2}$ function $f$,
respectively. 
\begin{equation}
\left|f_{L}-f_{Q}\right|=\left|(f_{L}-f)-(f_{Q}-f)\right|\leq\left|f_{L}-f\right|+\left|f_{Q}-f\right|,\label{eq:diff-linearinterp-quadratic-wls}
\end{equation}
where $\left|f_{L}-f\right|=\mathcal{O}(h^{2})$ and $\left|f_{Q}-f\right|=\mathcal{O}(h^{3})$
if $f$ is $C^{2}$, as per Lemma~\ref{lem:wls_accuracy}. Therefore,
$\left|f_{L}-f_{Q}\right|=\mathcal{O}(h^{2})$. If the mesh is composed
of quadrilaterals, we can replace linear interpolation with bilinear
interpolation, and the error bounds still hold.
\end{proof}
\begin{rem}
Even if we replace quadratic WLS with a different degree polynomial---be
it linear, cubic, or of a higher degree---Proposition~\ref{prop:smooth-error}
still holds true. This is because the error term in the proposition
is primarily dictated by the linear interpolation. However, we specifically
choose quadratic WLS for its balance of effectiveness and efficiency.
While linear WLS does not introduce oscillations, which are necessary
for identifying discontinuities, higher-degree WLS methods can increase
computational costs. As such, quadratic WLS presents the optimal degree
choice; it is the lowest-degree WLS method capable of introducing
oscillations, and yet it does so without significantly escalating
computational burden. This balance enables us to detect discontinuities
more efficiently.
\end{rem}

\subsection{Asymptotic bounds near discontinuities}

The analysis becomes more complex when a function has discontinuities.
The derivative of a $C^{0}$ function is a Heaviside function, and
the derivative of a Heaviside function at the discontinuity equates
to a Dirac delta function in a distribution sense. Although such a
generalized notion could be used within the context of Taylor series,
it is insufficient to employ the Taylor series alone to derive the
error bounds. Generalizing the analysis in Section~\ref{subsec:Weighted-least-squares-approxima},
we derive the error bounds as follows.
\begin{prop}
\label{prop:discon-error}In a neighborhood of $C^{0}$ and $C^{1}$
discontinuities of a piecewise-smooth function, the difference between
quadratic WLS and linear interpolation is $\mathcal{O}(1)$ and $\mathcal{O}(h)$,
respectively, where $h$ denotes an edge-length measure.
\end{prop}
\begin{proof}
First, we consider $C^{0}$ discontinuities, i.e., discontinuities
in function values. If $\boldsymbol{x}_{0}$ is close to a $C^{0}$
discontinuity, we can find a point $\boldsymbol{x}_{*}$ on the discontinuity
such that $\left|f(\boldsymbol{x}_{0})-f(\boldsymbol{x}_{*})\right|=\mathcal{O}(h)$.
For any point $\boldsymbol{x}$ in the vicinity of $\boldsymbol{x}_{0}$,
depending on whether $\boldsymbol{x}_{*}$ is on the same side of
the discontinuity as $\boldsymbol{x}_{0}$, we have $\left|f(\boldsymbol{x})-f(\boldsymbol{x}_{*})\right|=\mathcal{O}(h)$
or $\mathcal{O}(1)$. Hence, 
\begin{align*}
\left|f_{Q}(\boldsymbol{x})-f(\boldsymbol{x})\right| & =\left|(f_{Q}(\boldsymbol{x})-f(\boldsymbol{x}_{0}))-(f(\boldsymbol{x})-f(\boldsymbol{x}_{*}))+(f(\boldsymbol{x}_{0})-f(\boldsymbol{x}_{*}))\right|\\
 & \leq\left|f_{Q}(\boldsymbol{x})-f(\boldsymbol{x}_{0})\right|+\left|f(\boldsymbol{x})-f(\boldsymbol{x}_{*})\right|+\left|f(\boldsymbol{x}_{0})-f(\boldsymbol{x}_{*})\right|.
\end{align*}
Owing to the $\mathcal{O}(1)$ jump, $\max_{i,j}\left\{ \vert f_{i}-f_{j}\vert\right\} =\mathcal{O}(1)$.
The gradient and Hessian of the WLS fitting are bounded by $\tilde{\boldsymbol{g}}=\mathcal{O}(h^{-1})$
and $\tilde{\boldsymbol{H}}=\mathcal{O}(h^{-2})$, respectively. Therefore,
\[
\left|f_{Q}(\boldsymbol{x})-f(\boldsymbol{x}_{0})\right|=\left|\boldsymbol{h}^{T}\tilde{\boldsymbol{g}}+\frac{1}{2}\boldsymbol{h}^{T}\tilde{\boldsymbol{H}}\boldsymbol{h}\right|=\mathcal{O}(1),
\]
and in turn
\[
\left|f_{Q}(\boldsymbol{x})-f(\boldsymbol{x})\right|\leq\left|f_{Q}(\boldsymbol{x})-f(\boldsymbol{x}_{0})\right|+\mathcal{O}(1)=\mathcal{O}(1).
\]

Second, if $\boldsymbol{x}_{0}$ is close to a $C^{1}$ discontinuity,
i.e., discontinuities in derivatives, for any point $\boldsymbol{x}$
in the vicinity of $\boldsymbol{x}_{0}$, we can find a point $\boldsymbol{x}_{*}$
on the discontinuity such that $\left|f(\boldsymbol{x}_{0})-f(\boldsymbol{x}_{*})\right|=\mathcal{O}(h)$
and $\left|f(\boldsymbol{x})-f(\boldsymbol{x}_{*})\right|=\mathcal{O}(h)$.
Due to the discontinuities, $\max_{i,j}\left\{ \vert f_{i}-f_{j}\vert\right\} =\mathcal{O}(h)$.
Hence, the gradient and Hessian of the WLS fitting are $\tilde{\boldsymbol{g}}=\mathcal{O}(1)$
and $\tilde{\boldsymbol{H}}=\mathcal{O}(h^{-1})$, respectively, and
\[
\left|f_{Q}(\boldsymbol{x})-f(\boldsymbol{x}_{0})\right|=\left|\boldsymbol{h}^{T}\tilde{\boldsymbol{g}}+\frac{1}{2}\boldsymbol{h}^{T}\tilde{\boldsymbol{H}}\boldsymbol{h}\right|=\mathcal{O}(h),
\]
Therefore, 
\[
\left|f_{Q}(\boldsymbol{x})-f(\boldsymbol{x})\right|\leq\left|f_{Q}(\boldsymbol{x})-f(\boldsymbol{x}_{0})\right|+\left|f(\boldsymbol{x})-f(\boldsymbol{x}_{*})\right|+\left|f(\boldsymbol{x}_{0})-f(\boldsymbol{x}_{*})\right|=\mathcal{O}(h).
\]
\end{proof}
\begin{prop}
\label{prop:reach-discon-error}The difference between quadratic polynomial
fitting and linear interpolation reaches $\varTheta(1)$ and $\varTheta(h)$
at some points near $C^{0}$ and $C^{1}$ discontinuities, respectively. 
\end{prop}
\begin{proof}
The full proof follows a similar logic to Proposition~\ref{prop:discon-error}
but requires decomposing the function $f$ in terms of discontinuity
behavior and careful selection of suitable points $\boldsymbol{x}_{0}$
for the analysis.

Let us first consider a function $f$ with $C^{0}$ discontinuities.
For convenience, we can decompose it into two components: one with
and one without discontinuities, respectively. Specifically, we express
$f$ as $f=f^{0}+f^{1}$, where $f^{0}$ contains $C^{0}$ discontinuities
and $f^{1}$ is $C^{0}$ continuous. We represent the quadratic polynomial
fitting and linear interpolation of $f$ at a point $\boldsymbol{x}$
as $q(\boldsymbol{x}:f)$ and $l(\boldsymbol{x}:f)$, respectively.
Hence, we have
\[
q(\boldsymbol{x}:f)=q(\boldsymbol{x}:f^{0})+q(\boldsymbol{x}:f^{1})
\]
and
\[
l(\boldsymbol{x}:f)=l(\boldsymbol{x}:f^{0})+l(\boldsymbol{x}:f^{1}).
\]
Let us consider $\boldsymbol{x}_{0}\in e_{0}$, a point near $C^{0}$
discontinuities, such that the element $e_{0}$ does not cover any
$C^{0}$ discontinuities region, but the stencil of $\boldsymbol{x}_{0}$
does. The quadratic polynomial fitting at $\boldsymbol{x}_{0}$ can
be expressed as 
\[
q(\boldsymbol{x}_{0}:f)=\sum_{i=1}^{n}c_{i}f_{i},
\]
where $f_{i}=f(\boldsymbol{x}_{i})$ denotes the function value at
the stencil $\{\boldsymbol{x}_{i}\}_{i=1}^{n}$ of $\boldsymbol{x}_{0}$
and $c_{i}$ is related only to the position of $\boldsymbol{x}_{0}$
and its stencil $\{\boldsymbol{x}_{i}\}_{i=1}^{n}$. 

Assuming that the mesh is fine enough to contain only one $C^{0}$
discontinuity in the stencil, i.e., we can decompose $f$ in a small
neighborhood $\Lambda$ of $\boldsymbol{x}_{0}$ such that 
\[
f^{0}(\boldsymbol{x})=\begin{cases}
\lambda, & \textrm{if }\boldsymbol{x}\in\Gamma,\\
0, & \textrm{if }\boldsymbol{x}\in\Lambda\backslash\Gamma,
\end{cases}
\]
where $\lambda\neq0$ is a constant approximating the jump of the
$C^{0}$ discontinuity. Either $e_{0}\subset\Gamma$ or $e_{0}\subset\Lambda\backslash\Gamma$.
For the former case, we choose a new $\tilde{f}^{0}=\lambda-f^{0}$.
It is sufficient to set the non-zero region $\Gamma$ of $f^{0}$
to satisfy that $e_{0}\cap\Gamma=\emptyset$. Consequently, 
\[
l(\boldsymbol{x}_{0}:f^{0})=0.
\]
Because a $C^{0}$ discontinuity exists, $\{\boldsymbol{x}_{i}\}_{i=1}^{n}\cap\Gamma\neq\emptyset$.
Denoting $K=\{i|\boldsymbol{x}_{i}\in\Gamma,1\leq i\leq n\},$ we
get
\[
q(\boldsymbol{x}_{0}:f^{0})=\sum_{i\in K}c_{i}f_{i}^{0}=\lambda\sum_{i\in K}c_{i}.
\]
In a local neighborhood, $\Gamma$ is fixed. By moving $\boldsymbol{x}_{0}$
such that the cardinality of $K$ is one, $\#K=1$, we have $\sum_{i\in K}c_{i}\neq0$.
Consequently, 
\[
q(\boldsymbol{x}_{0}:f^{0})=\lambda\sum_{i\in K}c_{i}=\varTheta(1).
\]
According to Proposition~\ref{prop:discon-error}, 
\[
q(\boldsymbol{x}:f^{1})-l(\boldsymbol{x}:f^{1})=O(h),
\]
and therefore
\[
q(\boldsymbol{x}:f)-l(\boldsymbol{x}:f)=\varTheta(1)+O(h)=\varTheta(1).
\]

Secondly, if $\boldsymbol{x}$ is near a $C^{1}$ discontinuity, we
decompose $f$ in a local neighborhood of $\boldsymbol{x}$ as $f=f^{2}+f^{3}$,
where $f^{2}$ is $C^{0}$ continuous but contains $C^{1}$ discontinuities,
and $f^{3}$ is $C^{2}$ continuous. Assuming the mesh is fine enough
to have only one $C^{1}$ discontinuity in the stencil, we can find
a decomposition of $f$ so that
\[
f^{2}(x)=\begin{cases}
\Theta(h), & \boldsymbol{x}\in\Gamma\\
0, & \textrm{otherwise}
\end{cases},
\]
where $\Gamma$ is fixed in a local region and $e_{0}\cap\Gamma=\emptyset$.
Consequently, we have $l(\boldsymbol{x}_{0}:f^{2})=0$. In a similar
manner, by finding an $\boldsymbol{x}_{0}$ such that the cardinality
of $K=\{i|\boldsymbol{x}_{i}\in\Gamma,1\leq i\leq n\}$ is one, we
have 
\[
q(\boldsymbol{x}_{0}:f^{2})=\sum_{i\in K}c_{i}f_{i}^{2}=\Theta(h).
\]
According to Proposition~\ref{prop:smooth-error}, $q(\boldsymbol{x}:f^{3})-l(\boldsymbol{x}:f^{3})=O(h^{2})$,
and thus,
\[
q(\boldsymbol{x}:f)-l(\boldsymbol{x}:f)=\varTheta(h)+O(h^{2})=\varTheta(h).
\]
This completes the proof of the proposition.
\end{proof}
It is important to note that the bounds in Proposition~\ref{prop:discon-error}
represent upper limits. In practice, the differences may be significantly
smaller at certain points, so discontinuity indicators may yield false
negatives near such points, which we will address in Section~\ref{sec:Node-based-Markers}.

\subsection{\label{subsec:Precomputing-operator-for}Precomputing operator for
cell-based overshoot-undershoot indicators}

\subsubsection*{Overshoot-undershoot indicator}

We first build an operator to compute cell-based indicators for local
overshoots and undershoots, which are inspired by Proposition~\ref{prop:discon-error}.
For a cell $\sigma$, let $g_{\sigma,1}$ and $g_{\sigma,2}$ denote
the approximate values at the ``cell center'' of $\sigma$ from
linear interpolation and quadratic WALF, respectively. We compute
a value $\alpha_{\sigma}$ as the difference between $g_{\sigma,2}$
and $g_{\sigma,1}$, that is, 
\begin{equation}
\alpha_{\sigma}=g_{\sigma,2}-g_{\sigma,1}.\label{eq:cell-based-indicator}
\end{equation}
We refer to $\alpha_{\sigma}$ as the \emph{overshoot-undershoot}
(\emph{OSUS}) \emph{indicator} at $\sigma$, since its positive and
negative sign respectively indicates local overshoot and undershoot
of $g_{\sigma,2}$ at $\sigma$, and its magnitude indicates the level
of overshoot or undershoot. We compute the cell center by simply averaging
the nodes of $\sigma$. Note that this point is not necessarily the
centroid in general. We choose this point to compute $g_{\sigma,:}$
because $g_{\sigma,1}$ is essentially the average of corresponding
nodal values, and more importantly, $g_{\sigma,2}$ is prone to overshoot
and undershoot more significantly at this point than at points closer
to nodes.

Both $g_{\sigma,2}$ and $g_{\sigma,1}$ are weighted sums of nodal
values, and so is $\alpha_{\sigma}$. Consequently, the computation
of the $\alpha$ values for all the cells can be expressed as a sparse
matrix-vector multiplication. We refer to this sparse matrix as the
\emph{OSUS operator}. Each row of the operator corresponds to a cell
center, each column corresponds to a node in the mesh, and each nonzero
entry stores a nonzero weight. Specifically, we compute the quadratic
WALF using the 1.5-ring stencil at each node of the cell \citep{Jiao2011RHO},
and hence the nonzeros in each row for a cell $\sigma$ would correspond
to the union of the 1.5-rings of the nodes of $\sigma$. This RDI
operator only depends on the mesh and is independent of the function
values, so it can be pre-computed for a given mesh.

Its computation requires a data structure that supports computing
the $k$-ring neighborhood of nodes, and we implement it using an
array-based half-facet data structure \citep{DREJTAHF2014}. Occasionally,
for some poorly shaped meshes, the Vandermonde matrix from WALF may
be ill-conditioned. When this happens, we monitor the condition number
of the Vandermonde matrix and enlarge the stencils for ill-conditioned
systems.

\subsubsection*{Element-based threshold}

In accordance with Propositions~\ref{prop:smooth-error}, \ref{prop:discon-error},
and \ref{prop:reach-discon-error}, most of the smooth regions can
be filtered out by an element-wise threshold
\begin{equation}
\tau_{\sigma}=\max\left\{ \underbrace{C_{\ell}\delta f_{\ell}h_{\ell}^{0.5}}_{\tau_{\ell}},\underbrace{C_{g}\delta f_{g}h_{g}^{1.5}}_{\tau_{g}}\right\} ,\label{eq:threshold-element}
\end{equation}
where $\delta f_{\ell}$ represents the local range of $f$ within
the $k$-ring neighborhood for WALF reconstruction, $h_{\ell}$ is
the local average edge length in the local $uv$ coordinate system,
and $\delta f_{g}$ denotes the global range of function over the
mesh, and $h_{g}$ denotes a global measure of average edge length
in the $xyz$ coordinate system. $C_{\ell}$ and $C_{g}$ are two
parameters, which we determine empirically. All the elements with
$|\alpha_{\sigma}|>\tau_{\sigma}$ would be marked. However, some
elements at the peak might also be marked, leading to false positives.
We will further filter out these false positives using the node-based
indicator in Section~\ref{sec:Node-based-Markers}.

\section{Node-based Markers via Dual Thresholding\label{sec:Node-based-Markers}}

In Section~\ref{subsec:Precomputing-operator-for}, we defined an
element-based indicator to exclude most of the smooth region. However,
using solely the element-based indicator may lead to false positives,
particularly in the vicinity of local extremes on coarse meshes. To
mitigate this issue, we propose a \emph{dual-thresholding} approach
using node-based markers.

\subsection{Properties of $\alpha_{\sigma}$ near local extremes}

To derive more accurate indicators, we need a more precise estimation
of $\alpha_{\sigma}$. In the following, we derive the theory first
in $\mathbb{R}^{2}$ and then generalize it to surfaces.
\begin{prop}
\label{prop:alpha same sign plane}In the $C^{2}$ regions of a piecewise
smooth function, $\alpha_{\sigma}$ is negative near a local minimum
and positive near a local maximum on a sufficiently fine mesh of a
plane.
\end{prop}
\begin{proof}
Without loss of generality, we assume the mesh is composed of triangles.
Consider a function $f:\mathbb{R}^{2}\rightarrow\mathbb{R}$ defined
on a plane. For a triangular element $\sigma=\boldsymbol{x}_{1}\boldsymbol{x}_{2}\boldsymbol{x}_{3}$
within the $C^{2}$ region of $f$, we obtain the linear interpolation
$g_{\sigma,1}$ and the WALF reconstructed value $g_{\sigma,2}$ from
(\ref{eq:WALF}) as follows: 
\[
g_{\sigma,1}=\frac{1}{3}\sum_{i=1}^{3}f(\boldsymbol{x}_{i})\quad\textrm{and}\quad g_{\sigma,2}=\frac{1}{3}\sum_{i=1}^{3}f_{i}(\boldsymbol{x}_{0}),
\]
where $\boldsymbol{x}_{0}=\frac{1}{3}(\boldsymbol{x}_{1}+\boldsymbol{x}_{2}+\boldsymbol{x}_{3})$
is the cell center of $\sigma$ and $f_{i}$ is the degree-$2$ WLS
centered at $\boldsymbol{x}_{i}$. Hence, we obtain
\begin{align}
\alpha_{\sigma} & =g_{\sigma,2}-g_{\sigma,1}\nonumber \\
 & =\frac{1}{3}\sum_{i=1}^{3}f_{i}(\boldsymbol{x}_{0})-\frac{1}{3}\sum_{i=1}^{3}f(\boldsymbol{x}_{i})\nonumber \\
 & =\frac{1}{3}\sum_{i=1}^{3}(f_{i}(\boldsymbol{x}_{0})-f(\boldsymbol{x}_{0}))+f(\boldsymbol{x}_{0})-\frac{1}{3}\sum_{i=1}^{3}f(\boldsymbol{x}_{i})\nonumber \\
 & =f(\boldsymbol{x}_{0})-\frac{1}{3}\sum_{i=1}^{3}f(\boldsymbol{x}_{i})+O(h^{3}).\label{eq:alpha=00003Ddiff f}
\end{align}
The last equality above holds because of Lemma~\ref{lem:wls_accuracy}.
Let $(\xi,\eta)$ denote the natural coordinates of $\sigma$. Then,
\[
\boldsymbol{x}(\xi,\eta)=\boldsymbol{x}_{1}+\xi\boldsymbol{\mu}_{1}+\eta\boldsymbol{\mu}_{2},
\]
where $\boldsymbol{\mu}_{1}=\boldsymbol{x}_{2}-\boldsymbol{x}_{1}$,
$\boldsymbol{\mu}_{2}=\boldsymbol{x}_{3}-\boldsymbol{x}_{1}$ are
the two edges of $e$. The Taylor series expansion \citep{humpherys2017foundations}
of $f$ on $\boldsymbol{x}_{1}$ is 
\begin{equation}
f(\boldsymbol{x}(\xi,\eta))=f(\boldsymbol{x}_{1})+(\xi\boldsymbol{\mu}_{1}+\eta\boldsymbol{\mu}_{2})^{T}\nabla f+\frac{1}{2}(\xi\boldsymbol{\mu}_{1}+\eta\boldsymbol{\mu}_{2})^{T}\tilde{\boldsymbol{H}}(\xi\boldsymbol{\mu}_{1}+\eta\boldsymbol{\mu}_{2})+O(h^{3}),\label{eq:taylor tilde f}
\end{equation}
where $\tilde{\boldsymbol{H}}$ is the Hessian matrix of $f$ at $\boldsymbol{x}_{1}$.
Substituting the terms in (\ref{eq:alpha=00003Ddiff f}) with (\ref{eq:taylor tilde f}),
we obtain
\begin{align}
\alpha_{\sigma} & =\frac{1}{2}\left[(\frac{1}{3}\boldsymbol{\mu}_{1}+\frac{1}{3}\boldsymbol{\mu}_{2})^{T}\tilde{\boldsymbol{H}}(\frac{1}{3}\boldsymbol{\mu}_{1}+\frac{1}{3}\boldsymbol{\mu}_{2})-\frac{1}{3}\boldsymbol{\mu}_{1}^{T}\tilde{\boldsymbol{H}}\boldsymbol{\mu}_{1}-\frac{1}{3}\boldsymbol{\mu}_{2}^{T}\tilde{\boldsymbol{H}}\boldsymbol{\mu}_{2}\right]+O(h^{3})\nonumber \\
 & =-\frac{1}{18}(\boldsymbol{\mu}_{1}^{T}\tilde{\boldsymbol{H}}\boldsymbol{\mu}_{1}+\boldsymbol{\mu}_{2}^{T}\tilde{\boldsymbol{H}}\boldsymbol{\mu}_{2}+\boldsymbol{\mu}_{3}^{T}\tilde{\boldsymbol{H}}\boldsymbol{\mu}_{3})+O(h^{3}),\label{eq:alpha plane uhu}
\end{align}
where $\boldsymbol{\mu}_{3}\coloneqq\boldsymbol{\mu}_{2}-\boldsymbol{\mu}_{1}=\boldsymbol{x}_{3}-\boldsymbol{x}_{2}$
is the third edge of $e$. Note that $\boldsymbol{\mu}_{i}=O(h)$,
so the first term in (\ref{eq:alpha plane uhu}) is the dominant term.
The Hessian matrix is positive definite near a local minimum and negative-definite
near a local maximum, so $\alpha_{\sigma}$ will be negative and positive,
respectively, when $h$ is sufficiently small.
\end{proof}
In Section~\ref{subsec:Computing-node-based-oscillation}, we will
utilize the property above to differentiate the smooth region near
extremes from discontinuities with alternating signs of $\alpha_{\sigma}$.
To apply it to surface meshes, however, we need to generalize the
result to surfaces. We assume the vertices interpolate the piecewise
smooth surface.
\begin{prop}
\label{prop:alpha same sign general surface}Within the intersection
of $C^{2}$ regions of a piecewise smooth function and $G^{2}$ regions
of the surface, the $\alpha_{\sigma}$ are negative near a local minimum
and positive near a local maximum, provided that the mesh of the surface
is sufficiently fine. 
\end{prop}
\begin{proof}
Let us consider a function $f:\Omega\rightarrow\mathbb{R}$ defined
on a general surface $\Omega$. Within a $G^{2}$ region of $\Omega$,
suppose there is a triangular element $\sigma=v_{1}v_{2}v_{3}$. We
construct three smooth projections, $p_{k}:\mathbb{R}^{3}\rightarrow\Omega$,
from the normal plane of $v_{k}$ to a small neighborhood $\Phi_{k}$
of $v_{k}$ for $k=1,2,3$. We assume the mesh is fine enough so that
$\sigma\subset\Phi_{k},k=1,2,3$. Additionally, we construct another
projection $q_{k}:\mathbb{R}^{2}\rightarrow\mathbb{R}^{3}$ from $\mathbb{R}^{2}$
to the normal plane, keeping the distance preserved. That is, 
\[
\left\Vert \boldsymbol{x}_{1}-\boldsymbol{x}_{2}\right\Vert _{2}=\left\Vert q_{k}(\boldsymbol{x}_{1})-q_{k}(\boldsymbol{x}_{2})\right\Vert _{2},\forall\boldsymbol{x}_{1},\boldsymbol{x}_{2}\in\mathbb{R}^{2},1\leq k\leq3.
\]
Let $\boldsymbol{x}_{i}^{k}$ be the corresponding vertex of $v_{i}$
for $p_{k}$ and $q_{k}$ on $\mathbb{R}^{2}$. It then follows that
$p_{k}(q_{k}(\boldsymbol{x}_{i}^{k}))=v_{i}$ for $i=1,2,3$. The
WALF reconstructed value is then given by 
\[
g_{\sigma,2}=\frac{1}{3}\sum_{i=1}^{3}f_{i}(v_{0}),
\]
where $v_{0}$ is the cell center of $\sigma$, and $f_{i}$ is the
degree-$2$ WLS centered at $v_{i}$. Let $\tilde{f}_{k}(\boldsymbol{x})\coloneqq f_{k}(p_{k}(q_{k}(\boldsymbol{x})))$,
and $\alpha_{\sigma}$ is then given by
\begin{align}
\alpha_{\sigma} & =g_{\sigma,2}-g_{\sigma,1}\nonumber \\
 & =\frac{1}{3}\sum_{k=1}^{3}f_{k}(v_{0})-\frac{1}{3}\sum_{i=1}^{3}f(v_{i})\nonumber \\
 & =\frac{1}{3}\sum_{k=1}^{3}(\tilde{f}_{k}(\boldsymbol{x}_{0}^{k})-\frac{1}{3}\sum_{i=1}^{3}\tilde{f}_{k}(\boldsymbol{x}_{i}^{k}))-\frac{1}{3}\sum_{i=1}^{3}(f(v_{i})-\frac{1}{3}\sum_{k=1}^{3}\tilde{f}_{k}(\boldsymbol{x}_{i}^{k})),\label{eq:alpha in general surface}
\end{align}
where $\boldsymbol{x}_{0}^{k}=\frac{1}{3}(\boldsymbol{x}_{1}^{k}+\boldsymbol{x}_{2}^{k}+\boldsymbol{x}_{3}^{k}),\forall1\leq k\leq3$.
Since the vertices are on the piecewise smooth surface, we can apply
Lemma~\ref{lem:wls_accuracy} and conclude that the second term in
(\ref{eq:alpha in general surface}) is $O(h^{3})$.

Following the argument as in Proposition~(\ref{prop:alpha same sign plane}),
we then get:
\begin{equation}
\tilde{f}_{k}(\boldsymbol{x}_{0}^{k})-\frac{1}{3}\sum_{i=1}^{3}\tilde{f}_{k}(\boldsymbol{x}_{i}^{k})=-\frac{1}{18}\sum_{i=1}^{3}(\boldsymbol{\mu}_{i}^{k})^{T}\tilde{\boldsymbol{H}_{k}}\boldsymbol{\mu}_{i}^{k}+O(h^{3}),\label{eq:alpha surface uhu}
\end{equation}
where $\tilde{\boldsymbol{H}_{k}}$ is the Hessian matrix for $\tilde{f}_{k}$
at $\boldsymbol{x}_{k}^{k}$ and $\boldsymbol{\mu}_{i}^{k}$ are the
edges for the triangle in the domain of $\tilde{f}_{k}$. The local
extremes of $f$ in $\Omega$ correspond to local extremes of $\tilde{f}_{k}$
in $\mathbb{R}^{2}$. Hence, the $\tilde{\boldsymbol{H}}_{k}$ are
either all positive definite or all negative definite, at the same
time. Therefore, the proposition has been proved.
\end{proof}

\subsection{Computing node-based oscillation indicators \label{subsec:Computing-node-based-oscillation}}

\subsubsection*{Definition of oscillation indicator}

Given the nodal values of a function $f$, we first calculate of the
cell-based $\alpha$ values by multiplying the OSUS operator with
a vector composed of nodal function values. However, since the $\mathcal{O}(1)$
and $\mathcal{O}(h)$ estimates in Proposition~\ref{prop:discon-error}
merely provide the upper bounds of $\alpha_{\sigma}$, these values
may be arbitrarily small at some isolated cells near discontinuities.
As a result, the OSUS indicators alone are insufficiently reliable
for detecting discontinuities. To enhance robustness, we further process
the OSUS indicators to generate node-based \emph{oscillation indicators}. 

Specifically, for a given node $v$, we compute its oscillation indicator
from the adjacent nodes using the formula:
\begin{equation}
\beta_{v}=\frac{\sum_{\sigma\ni v}w_{\sigma}\vert\alpha_{\sigma}-\bar{\alpha}_{v}\vert}{\left(\sum_{\sigma\ni v}w_{\sigma}\right)(\left|\bar{\alpha}_{v}\right|+\epsilon_{\beta}\max_{\sigma\ni v}|\alpha_{\sigma}|)+\epsilon_{\text{min}}}.\label{eq:node-based-indicator}
\end{equation}
Here, $\bar{\alpha}_{v}$ denotes a weighted average of the $\alpha$
values for the cells incident on $v$, i.e.,
\[
\bar{\alpha}_{v}=\sum_{\sigma\ni v}w_{\sigma}\alpha_{\sigma}/\sum_{\sigma\ni v}w_{\sigma},
\]
where the weights may be unit weights as \citep{li2020wls} or area-based
weights. The denominator in (\ref{eq:node-based-indicator}) comprises
safeguards (namely, the second and third terms) against division by
a value that is too small, which may happen if the function is locally
linear.  $\epsilon_{\beta}$ is a small number, which can be set
to $10^{-3}$ in practice. The last term $\text{\ensuremath{\epsilon_{\text{min}}}}$
in the denominator, which is the smallest positive floating-point
number (approximately $2.2\times10^{-308}$ for double-precision floating-point
numbers), serves as an additional safeguard against division by zero
if the input function is a constant. It should be noted that $\beta_{v}$
is non-dimensional and is independent of the function value and the
mesh scale.

From both a practical and numerical perspective, the definition of
$\beta_{v}$ is important. Computationally, given the element-based
$\alpha$ values, it is efficient to calculate the node-based $\beta$
values. Numerically, the definition of $\beta_{v}$ warrants justification.
Suppose that $f$ is sufficiently nonlinear such that $\beta_{v}\approx\sum_{\sigma\ni v}w_{\sigma}\vert\alpha_{\sigma}-\bar{\alpha}_{v}\vert/\left|\sum_{\sigma\ni v}w_{\sigma}\alpha_{\sigma}\right|$.
Firstly, this quantity no longer depends on $h$. Secondly, the numerator
of $\beta_{v}$ captures the variation of $\alpha_{\sigma}$, and
the denominator further intensifies this variation when $\alpha_{\sigma}$
alternates in sign, particularly near $C^{0}$ discontinuities. These
insights will be useful in the development of an effective thresholding
strategy. 

\subsubsection*{Derivation of node-based markers via dual thresholding}

As mentioned in Section~\ref{subsec:Related-work}, Gibbs phenomena
may occur for degree-$2$ polynomial fitting near $C^{0}$ and $C^{1}$
discontinuities, leading to oscillations. Linear interpolation would
not introduce any $O(1)$ oscillation. Hence, $\alpha_{\sigma}$ would
be positive in the overshoot region and negative in the undershoot
region. Consequently, $\sum_{\sigma\ni v}w_{\sigma}\alpha_{\sigma}$
would largely cancel out and be close to $0$ in comparison to $\sum_{\sigma\ni v}w_{\sigma}|\alpha_{\sigma}|$.
As a result, $\beta_{v}$ would be large near the discontinuities.
In contrast, Proposition~\ref{prop:alpha same sign general surface}
suggests that all the $\alpha_{\sigma}$ around $v$ would have the
same sign near a local extremum of $f$ in the smooth region, resulting
in a small $\beta_{v}$. In terms of implementation, we use the 1-ring
neighborhood to evaluate $\beta$ from $\alpha$. While using a larger
neighborhood may improve the precision of the indicator, it will increase
computational cost.

To distinguish between discontinuities, we use threshold $\kappa$
for $\beta_{v}$ and devise a \emph{dual-thresholding} strategy. Specifically,
we mark a node $v$ as discontinuity if it has a large $\beta$ value
and at least one of its incident elements has a large $\alpha$ value,
i.e., 
\[
\text{\ensuremath{\eta}}_{v}=\begin{cases}
1 & \text{if }\ensuremath{\beta_{v}>\kappa}\text{ and }\exists e\ni v\text{ s.t. }\alpha_{\sigma}>\tau_{\sigma}\\
0 & \text{otherwise}
\end{cases},
\]
where $\kappa$ is the threshold for $\beta_{v}$ and $\tau_{\sigma}$
is the threshold in Section~\ref{subsec:Precomputing-operator-for}.

The threshold $\kappa$ is determined empirically, selected from the
interval $[0.2,0.5]$ for $C^{1}$ discontinuities and $[1,2]$ for
$C^{0}$ discontinuities. A lower $\kappa$ value will result in a
more sensitive indicator but may introduce a higher number of false
positives. On the other hand, a higher $\kappa$ value will yield
a stricter indicator, potentially leading to an increase in false
negatives. Users can adjust $\kappa_{0}$ and $\kappa_{1}$ for $C^{0}$
and $C^{1}$ discontinuities, respectively, according to their specific
tolerance of false positives and false negatives.

\subsubsection*{Area weights for non-uniform meshes}

On nonuniform meshes, we observe that utilizing unit weights, as in
\citep{li2020wls} may result in many false positives. For instance,
let us assume we have a coarse element $\sigma_{0}$ and several fine
elements $\sigma_{k},k=1,2,...,n$ around $v$ at the intersection
of $C^{2}$ continuous region of $f$ and $G^{2}$ regions of the
surface. Assume $\left\Vert \boldsymbol{u}_{i}^{0}\right\Vert =O(h)$
and $\left\Vert \boldsymbol{u}_{i}^{k}\right\Vert =o(h)$ for all
$1\leq k\leq n$, where $\boldsymbol{u}_{i}^{k}$ are the edges of
element $\sigma_{k}$ for all $0\leq k\leq n,1\leq i\leq3$. Based
on (\ref{eq:alpha in general surface}) and (\ref{eq:alpha surface uhu}),
we obtain $\alpha_{\sigma_{0}}=O(h^{2})$ and $\alpha_{\sigma_{k}}=o(h^{2})$
for all $1\leq k\leq n$. With unit weights, we get $\bar{\alpha}_{v}\approx\frac{1}{n+1}\alpha_{\sigma_{0}}$,
which leads to
\[
\beta_{v}\approx\frac{\sum_{k=0}^{n}|\alpha_{\sigma_{k}}-\bar{\alpha}_{v}|}{|\sum_{k=0}^{n}\alpha_{\sigma_{k}}|}\approx\frac{|\alpha_{\sigma_{0}}-\bar{\alpha}_{v}|+\sum_{k=1}^{n}|\bar{\alpha}_{v}|}{|\alpha_{\sigma_{0}}|}\approx\frac{2n}{n+1}.
\]
In this case, $v$ could erroneously be identified as a $C^{1}$ discontinuity
node with such a large $\beta_{v}$. 

To mitigate such false positives, we employ area weights. Given $w_{\sigma_{0}}=O(h^{2})$
and $w_{\sigma_{k}}=o(h^{2})$ for all $1\leq k\leq n$, we have $\bar{\alpha}_{v}\approx\alpha_{\sigma_{0}}$,
yielding
\[
\beta_{v}\approx\frac{w_{\alpha_{0}}|\alpha_{\sigma_{0}}-\bar{\alpha}_{v}|+\sum_{k=1}^{n}w_{\alpha_{k}}|\bar{\alpha}_{v}|}{|w_{\alpha_{0}}\alpha_{\sigma_{0}}|}=o(1).
\]
As a result, $\beta_{v}$ is close to $0$, thereby preventing the
false positive.

\subsection{\label{subsec:Generalization-to-surfaces}Generalization to surfaces
with geometric discontinuities}

Our preceding analysis assumed $G^{2}$ continuity of the surface,
and hence is not applicable when dealing with geometric discontinuities
in normals and curvatures. To accommodate these geometric discontinuities
during the detection of numerical discontinuities, we employ virtual
splitting as presented in \citep{li2019compact} along the ridge curves.
Virtual splitting makes the twin half-edges along ridge curves appear
disconnected in the half-facet data structure and in turn, creates
a separate connected component for each smooth region. As a result,
the stencils in the WLS computations will not include points from
different sides of geometric discontinuities. In this work, we focus
on discontinuities in normal directions, such as sharp ridges and
corners. For instance, after the virtual splitting of a cylinder,
the two bases and the lateral surface become three distinct connected
components. We then separately run our RDI on these surface patches
with boundaries. Without virtual splitting, using RDI might introduce
many false positives and negatives, as we demonstrate in \ref{sec:virtual split and geometric discontinuities}.

Close to the geometric boundary, whether originally present or created
by virtual splitting, we observed that a 3-ring neighborhood provides
adequate rows in the Vandermonde matrix for stability while computing
the degree-2 WLS in the OSUS operator. Similarly, we could use a 2-ring
neighborhood while computing $\beta_{v}$ for a vertex on the boundary.
These adjustments aid in providing accurate indicator results, as
demonstrated in Section~\ref{subsec:Numerical-results-with-sharp-features}.

\subsection{Summary of algorithm}

We outline the complete algorithm for RDI. Our approach consists of
three stages: (1) computation of element-based discontinuity indicators,
denoted as $\alpha$ values; (2) calculation of node-based indicators,
referred to as $\beta$ values; and (3) acquisition of node-based
discontinuity markers using a novel dual-thresholding strategy.

Our design principle aims to separate offline preprocessing from online
computation, as well as separate the data structures. Given a mesh,
the offline segment involves computing the OSUS operator, which is
performed only when a new mesh is employed. Once the OSUS operator
is generated, we execute the rapid online segment, as outlined in
Algorithm~\ref{alg:Robust-Discontinuities-Indicator}, to detect
discontinuities with varying function values on the mesh.

\begin{algorithm}
\begin{algorithmic}[1]
\Require{A surface mesh $\Gamma=(V,E)$ containing the coordinate of vertices $V$ and connectivity table for elements $E$, the OSUS operator $\bold{O}$ generated from $\Gamma$, function values $\bold{f}$ on all the nodes, two node-based thresholds $\kappa_{0},\kappa_{1}$ for $C^{0}$ and $C^{1}$ discontinuities, $\kappa_{0}>\kappa_{1}$}
\Ensure{A vector of integer $\bold{I}$ indicating discontinuities on all the nodes}
\State{$\boldsymbol{\alpha} \leftarrow \boldsymbol{Of}$}
\State{$I_{v} \leftarrow 0, \forall v\in V$}
\ForAll{element $\sigma\in E$} \Comment{Pre-filtering using $\alpha$}
	\State{Compute the weight $w_{\sigma}$}
	\State{Compute the threshold $\tau_{\sigma}$ in \eqref{eq:threshold-element}}
	\If{$|\alpha_{\sigma}|>\tau_{\sigma}$}
		\State{$I_{v} \leftarrow 1, \forall v \in \sigma$}
	\EndIf
\EndFor		
\ForAll{vertex $v \in V$} \Comment{Resolve false positives}
	\If{$I_{v}=1$}
		\State{Compute $\beta_{v}$ in \eqref{eq:node-based-indicator}}
		\If{$\beta_{v}>\kappa_{0}$}
			\State{$I_{v} \leftarrow 2$}
		\ElsIf{$\beta_{v} \leq \kappa_{1}$}
			\State{$I_{v} \leftarrow 0$}
		\EndIf
	\EndIf
\EndFor
\end{algorithmic}

\caption{Robust Discontinuities Indicator\label{alg:Robust-Discontinuities-Indicator}}

\end{algorithm}

To analyze the time complexity of the algorithm, let us denote the
number of vertices as $N=|V|$ and the number of elements as $M=|E|$
for the complete virtual splitting surface. The offline segment, which
generates the sparse OSUS operator, incurs a time complexity of $O(N+M)$,
counted only once for each mesh.

As for the online part, the complexity of sparse matrix-vector multiplication
to compute $\boldsymbol{\alpha}$ is $O(M)$. The pre-filtering using
$\alpha$ also has a time complexity of $O(M)$. The process of resolving
false positives possesses a time complexity of $O(K)$, where $K$
is the count of elements that pass pre-filtering. When applying a
new function on the same mesh, another online part with complexity
$O(M)$ is required. For detecting discontinuities of $l$ functions
on the same mesh, the overall time complexity is $O(N+(l+1)M)$.

\section{Numerical Results\label{sec:Numerical-Results}}

In this section, we report several numerical experiments conducted
with RDI, including comparisons on planes, evaluations on general
surfaces with sharp features, and the use case of remapping nodal
values between meshes for discontinuous functions. We have implemented
our algorithms in MATLAB and converted the code into C++ using MATLAB
Coder. The numerical experiments were conducted using MATLAB R2022b
on Linux. Since one of the key applications of RDI is remapping in
climate modeling, which tends to employ various representations of
the sphere, some of our experiments use spherical meshes.

\subsection{Experimentation on unit spheres with nonuniform meshes\label{subsec:Experiments-non-uniform mesh}}

To evaluate our discontinuity indicators, we employ two piecewise
smooth functions, ``interacting waves'' and ``crossing waves,''
which were previously used in \citep{li2020wls}. The definitions
of these functions are as follows:
\begin{equation}
f_{1}\left(\theta,\varphi\right)=\begin{cases}
\begin{array}{c}
1\\
1.7-2.52\theta/\pi\\
0.44\\
0.24\\
0.12
\end{array} & \begin{array}{c}
0\leq\theta<5\pi/18\\
5\pi/18\leq\theta<\pi.\\
\pi/2\leq\theta<13\pi/18\\
13\pi/18\leq\theta<9\pi/10\\
9\pi/10\leq\theta\leq\pi
\end{array}\end{cases},\label{eq:f3}
\end{equation}
and
\begin{equation}
f_{2}\left(\theta,\varphi\right)=-1000+2000\textrm{sign}(\cos(\varphi))\begin{cases}
1/2 & 0\leq\theta<\pi/4\\
-4\left(\theta/\pi-1/2\right) & \pi/4\le\theta<\pi/2\\
4\left(\theta/\pi-1/2\right) & \pi/2\le\theta<3\pi/4\\
1 & 3\pi/4\le\theta<7\pi/8\\
-64\theta^{2}/\pi^{2}+112\theta/\pi-48 & 7\pi/8\leq\theta\leq\pi
\end{cases}.\label{eq:f4}
\end{equation}
Note that $f_{1}$ has $C^{0}$ and $C^{1}$ discontinuities, whereas
$f_{2}$ has $C^{0}$, $C^{1}$, and $C^{2}$ discontinuities. Figure~\ref{fig:Function-values}
displays the color maps of the two functions on two regionally refined
meshes. In particular, $f_{1}$ is shown on the dual of a spherical
centroidal Voronoi tessellations (SCVT) mesh \citep{ju2011voronoi},
and $f_{2}$ is shown on a cubed-sphere mesh. The meshes in Figure~\ref{fig:Function-values}
are relatively coarse. For ease of visualization, we intersect the
sphere with planes and will depict the values on the intersection
curves henceforth.

\begin{figure}
\subfloat[$f_{1}$ on the dual of a coarse regionally refined SCVT with the
intersection plane $x+y=0$.]{\includegraphics[width=0.48\columnwidth]{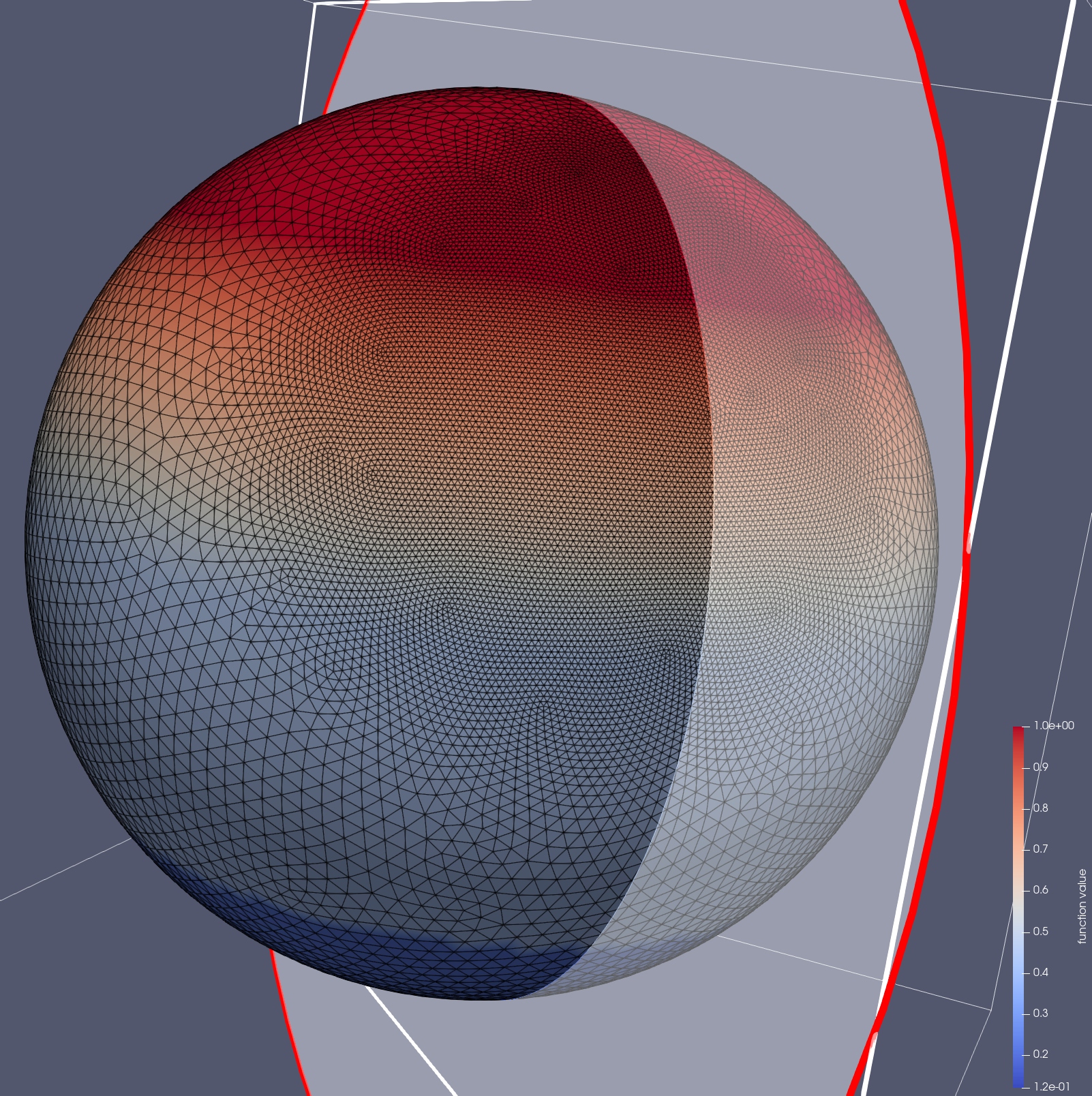}

}$\quad$\subfloat[$f_{2}$ on a coarse regionally refined cubed-sphere with the intersection
plane $y-2x=0$.]{\includegraphics[width=0.48\columnwidth]{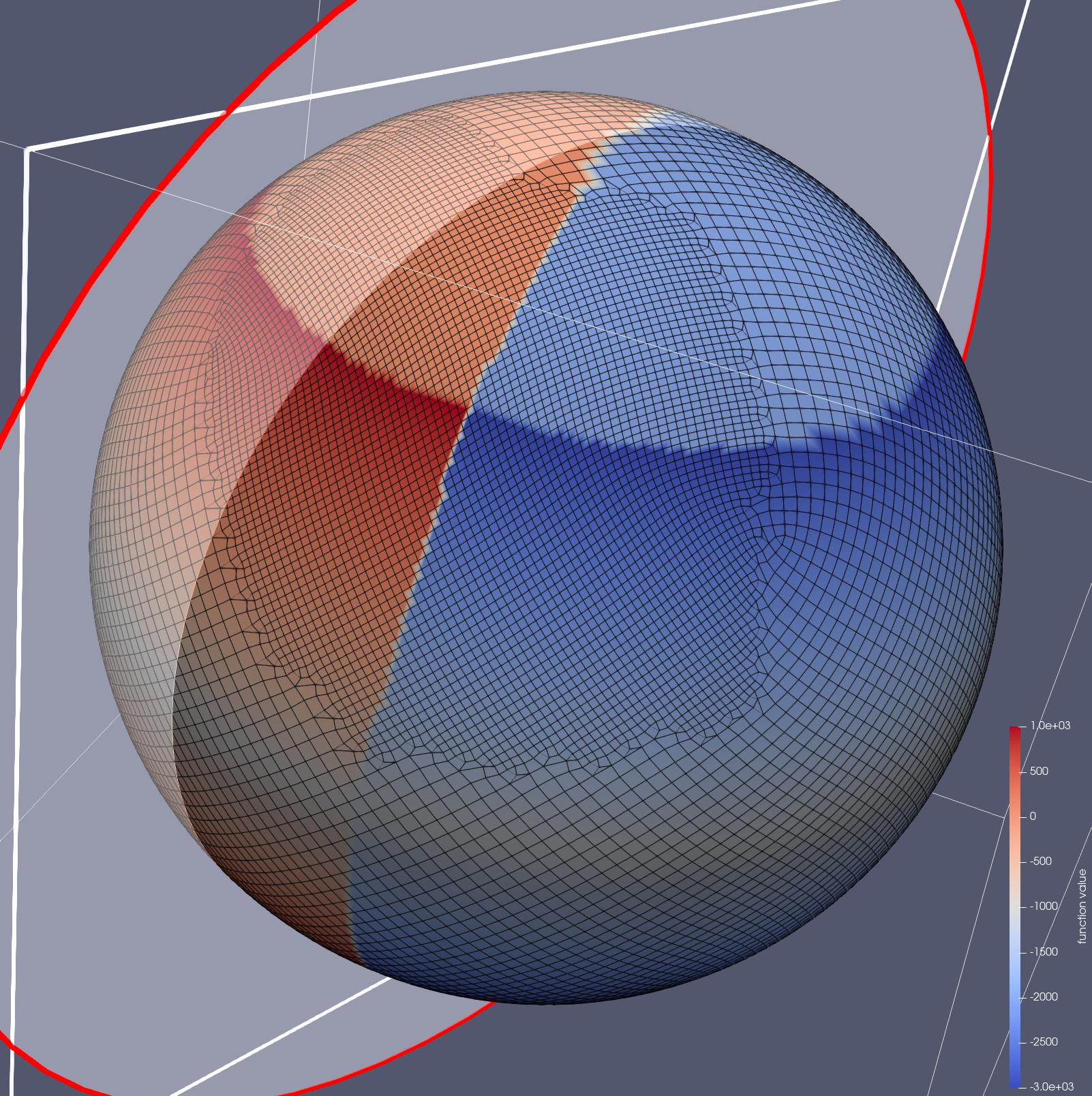}

}

\caption{\label{fig:Function-values}Function values on coarse regionally refined
meshes and the intersection planes.}
\end{figure}
\begin{figure}
\subfloat[$f_{1}$ and corresponding $\beta$ values.]{\begin{centering}
\includegraphics[width=0.48\columnwidth]{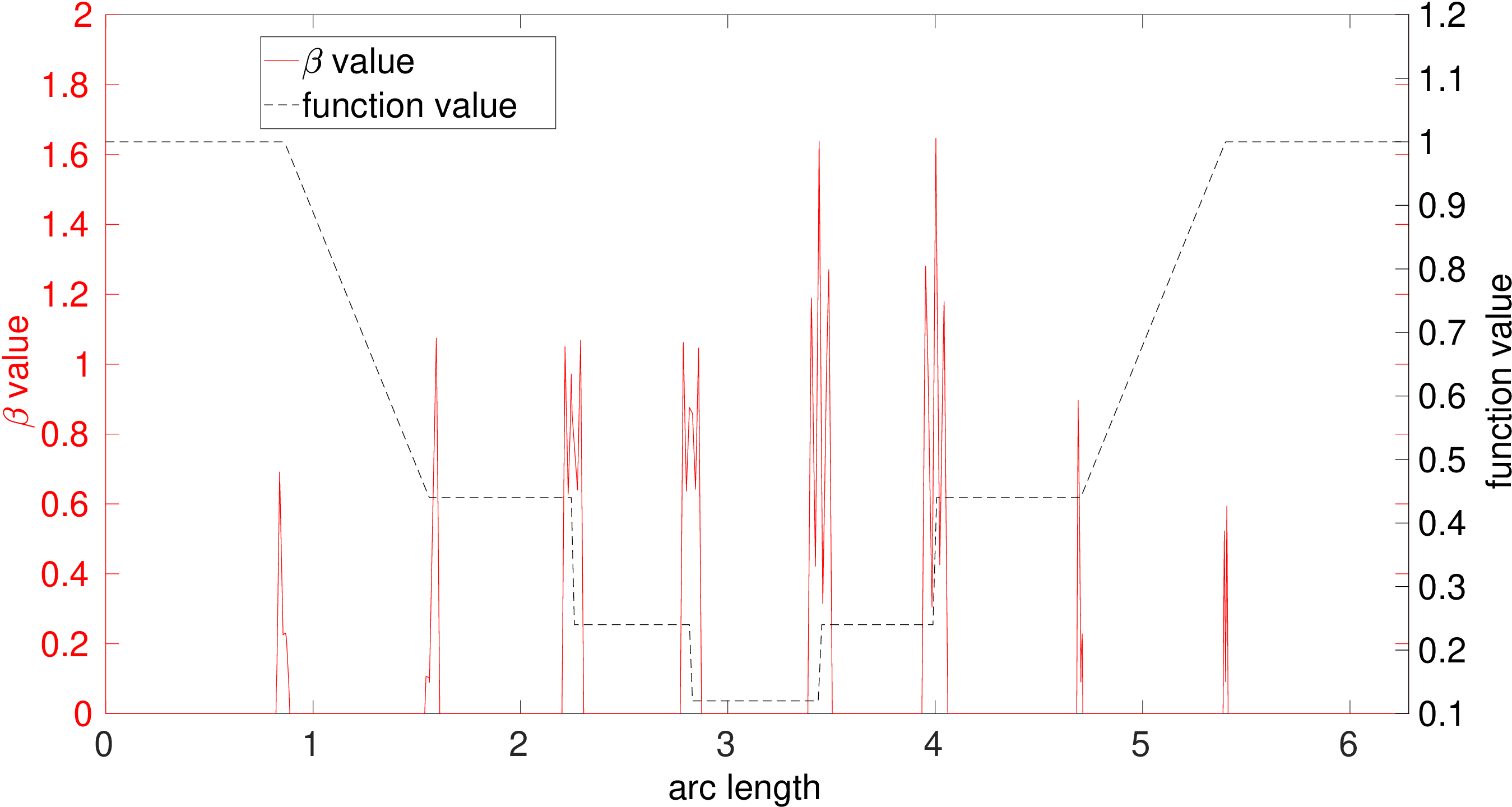}
\par\end{centering}
}\hfill\subfloat[$f_{2}$ and corresponding $\beta$ values.]{\begin{centering}
\includegraphics[width=0.48\columnwidth]{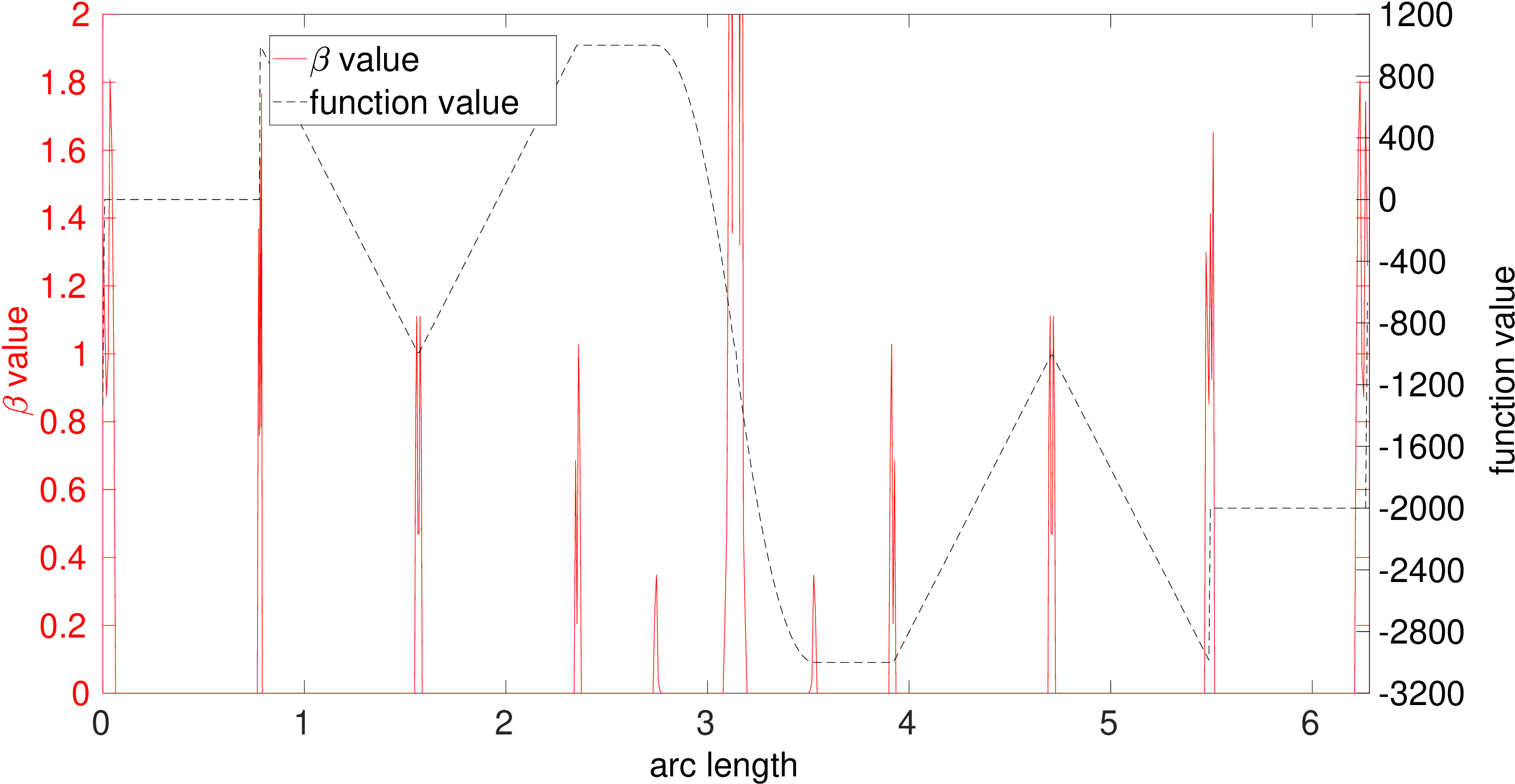}
\par\end{centering}
}

\caption{\label{fig:Example-node-based-indicator-1}Functions $f_{1}$ and
$f_{2}$ along with their corresponding $\beta$ values along the
cross section of finer regionally refined meshes with intersection
planes.}
\end{figure}

Figure~\ref{fig:Example-node-based-indicator-1} depicts the function
values and $\beta$ values along the intersection curves on the finer
meshes, which have edge lengths of approximately one fourth of those
seen in Figure~\ref{fig:Function-values}. It is evident that the
$\beta_{v}$ values can effectively distinguish the discontinuities
from the smooth regions. Figure~\ref{fig:Detected-discontinuity-cells}
shows the nodes where $\beta_{v}\geq0.5$. These nodes correspond
to the regions with $C^{0}$ and $C^{1}$ discontinuities. Hence,
we choose the threshold $\kappa=0.5$ for detecting discontinuities,
as previously mentioned in Section~\ref{subsec:Computing-node-based-oscillation}.
Reducing $\kappa$ to a smaller value may allow for the identification
of some $C^{2}$ discontinuities on coarse meshes but would also increase
the risk of false positiveness.

\begin{figure}
\subfloat[Detected discontinuity cells of $f_{1}$ on level-3 regionally refined
SCVT.]{\includegraphics[width=0.48\columnwidth]{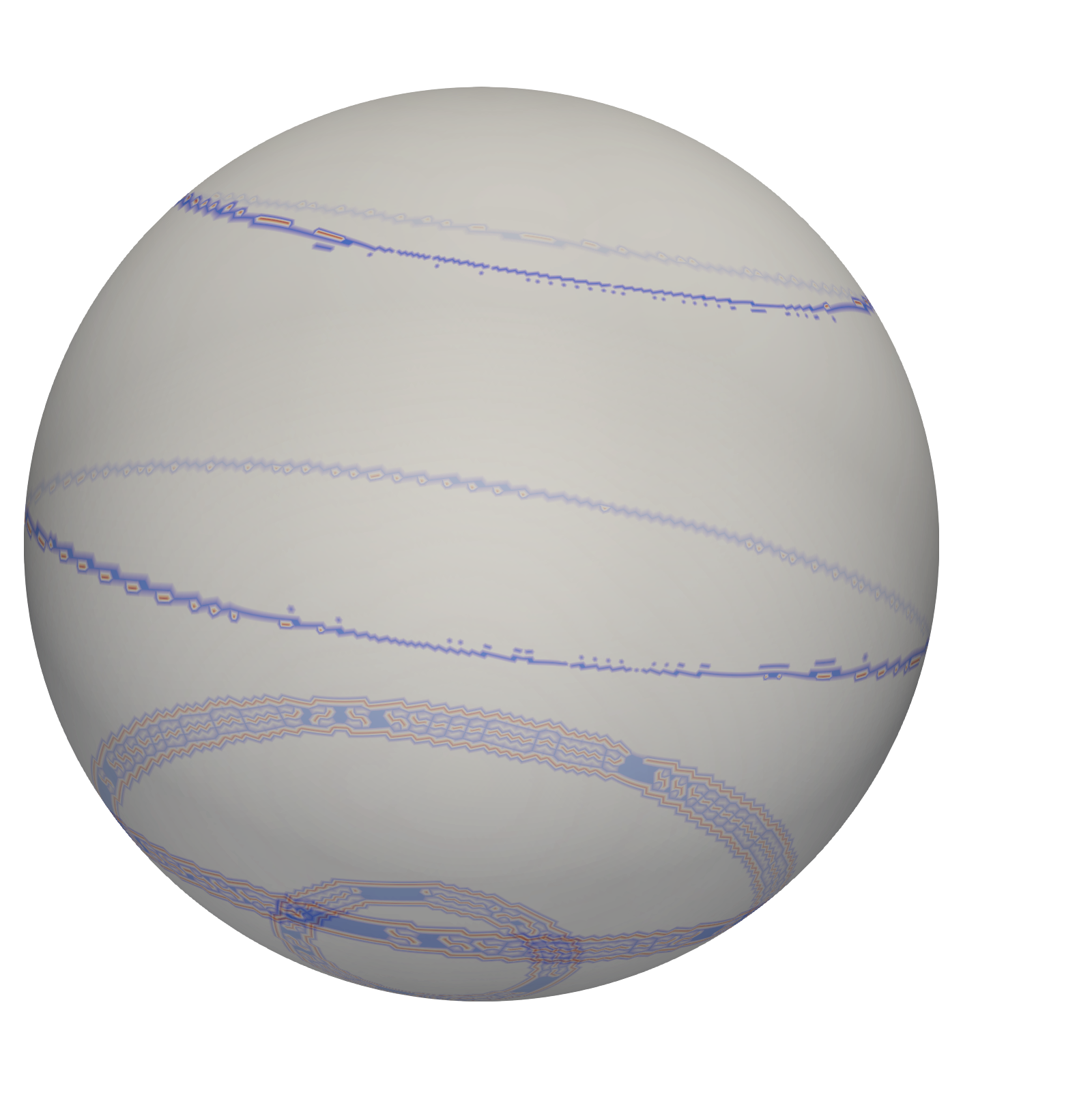}

}$\quad$\subfloat[Detected discontinuity cells of $f_{2}$ on level-3 regionally refined
cubed-sphere mesh.]{\includegraphics[width=0.48\columnwidth]{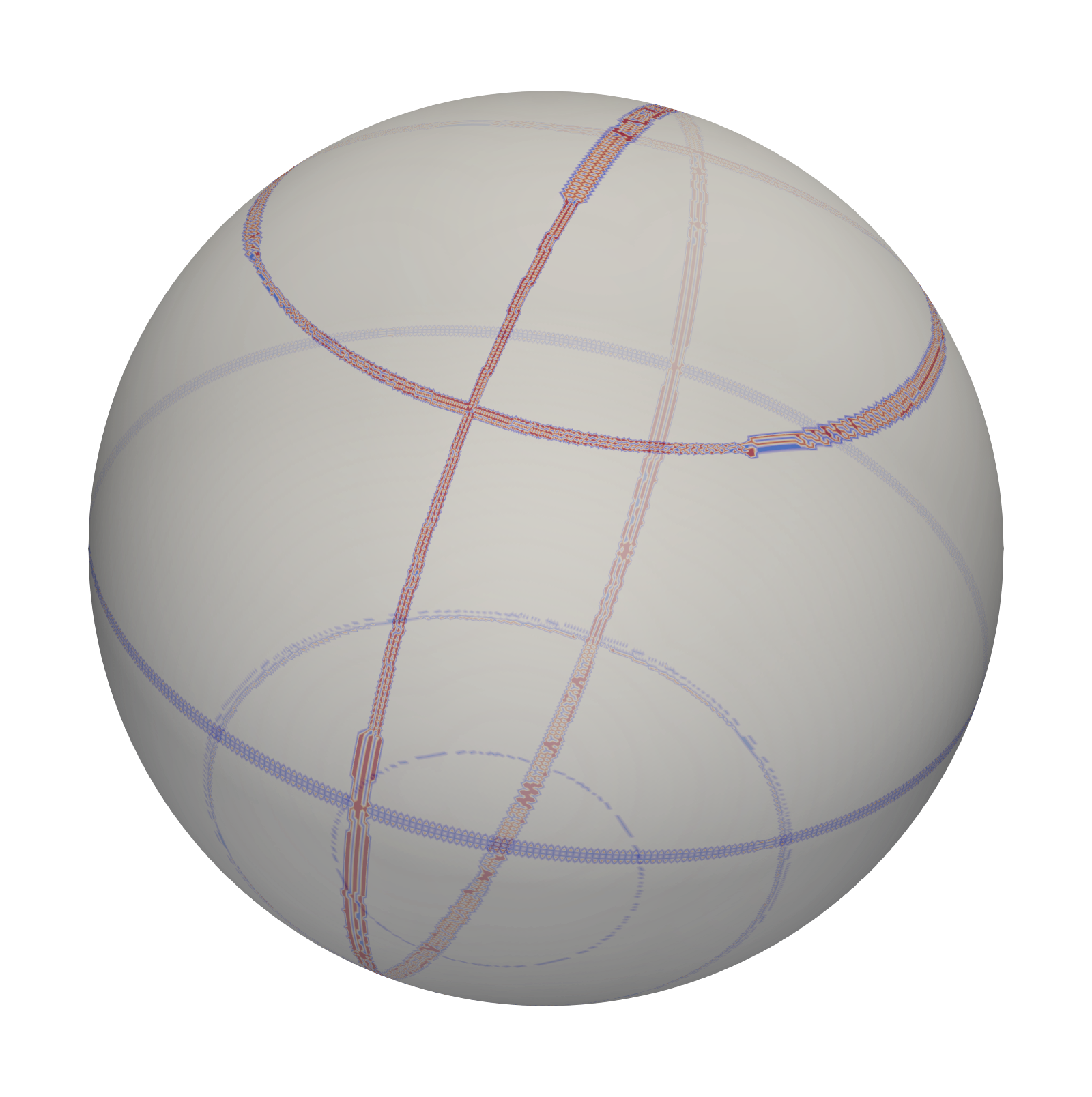}

}

\caption{Detected discontinuity cells (the red region) on regionally refined
meshes. \label{fig:Detected-discontinuity-cells}}
\end{figure}

\subsection{Comparison with minmod edge detection\label{subsec:Comparison-with-minmod}}

As alluded to in Section~\ref{subsec:Weighted-least-squares-approxima},
RDI is closely related to WLS-ENO, which is in turn related to limiters
in solving hyperbolic PDEs, such as minmod (see, e.g., \citep{leveque1992numerical}).
Therefore, we compare RDI with the minmod edge detection (MED) \citep{archibald2005polynomial,saxena2009high}.
Since MED is only defined in 2D, we conduct the comparison on planes. 

First, we use two test functions, $f_{3}$ and $f_{4}$, from \citep{archibald2005polynomial},
both of which contain only $C^{0}$ discontinuities. Specifically,
function $f_{3}$ is given by 
\[
f_{3}(x,y)=\begin{cases}
xy+\cos(2\pi x^{2})-\sin(2\pi x^{2}), & \textrm{if }x^{2}+y^{2}\leq\frac{1}{4},\\
10x-5+xy+\cos(2\pi x^{2})-\sin(2\pi x^{2}), & \textrm{if }x^{2}+y^{2}>\frac{1}{4},
\end{cases}
\]
for $-1\leq x,y\leq1$, and $f_{4}$ is the gray level of the Shepp-Logan
phantom \citep{shepp1974fourier}. 

For comparison with \citep{archibald2005polynomial}, we generated
$16,384$ random points on $[-1,1]\times[-1,1]$. As RDI uses a mesh
data structure for computing the neighborhood, we computed the Delaunay
triangulation of these points when applying RDI. To visualize the
detected discontinuities, we computed an approximation of the local
jump function
\[
\delta f(x)=\begin{cases}
\max_{R(x)}f(x)-\min_{R(x)}f(x), & \textrm{if }x\in D(f),\\
0, & \textrm{if }x\notin D(f),
\end{cases}
\]
where $D(f)$ is the discontinuity region detected by RDI, and $R(x)$
is the $1$-ring neighborhood of the point $x$ on the sampled mesh.
Figure~\ref{fig:Original-function-and} shows the local jump function,
which is similar to Figures~4.2 and 4.3 of MED \citep{archibald2005polynomial}.
However, the detected discontinuities from MED appear to be thinner
than those of RDI, since RDI uses larger stencils.

\begin{figure}
\subfloat[$f_{3}$]{\includegraphics[width=0.48\columnwidth]{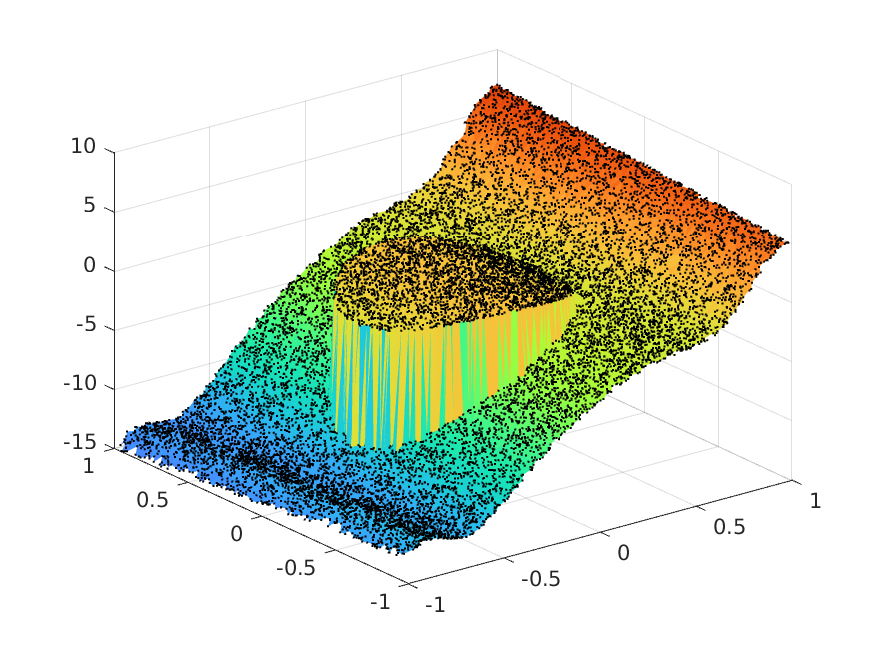}

}\subfloat[$\delta f_{3}$]{\includegraphics[width=0.48\columnwidth]{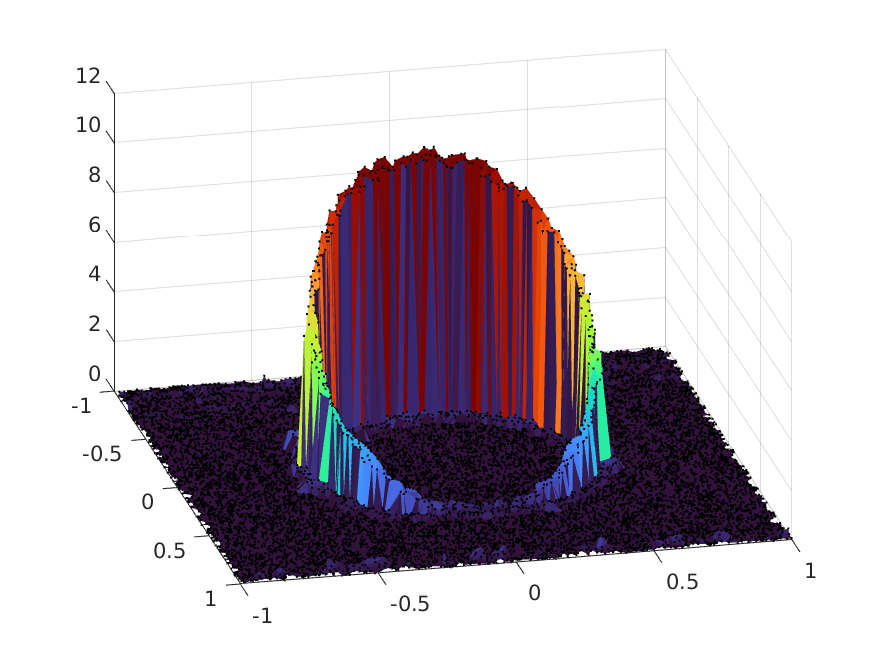}

}

\subfloat[$f_{4}$]{\includegraphics[width=0.48\columnwidth]{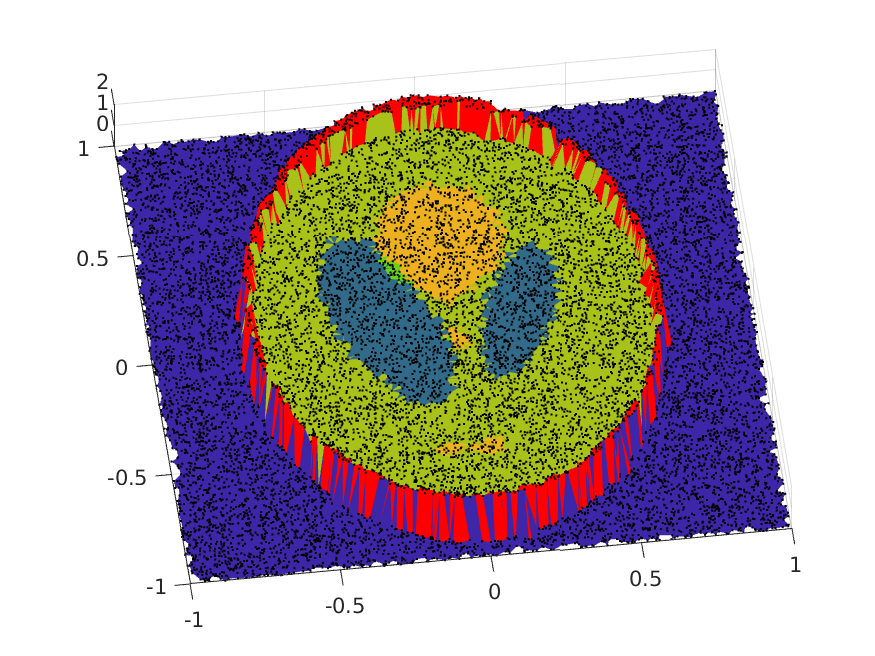}

}\subfloat[$\delta f_{4}$]{\includegraphics[width=0.48\columnwidth]{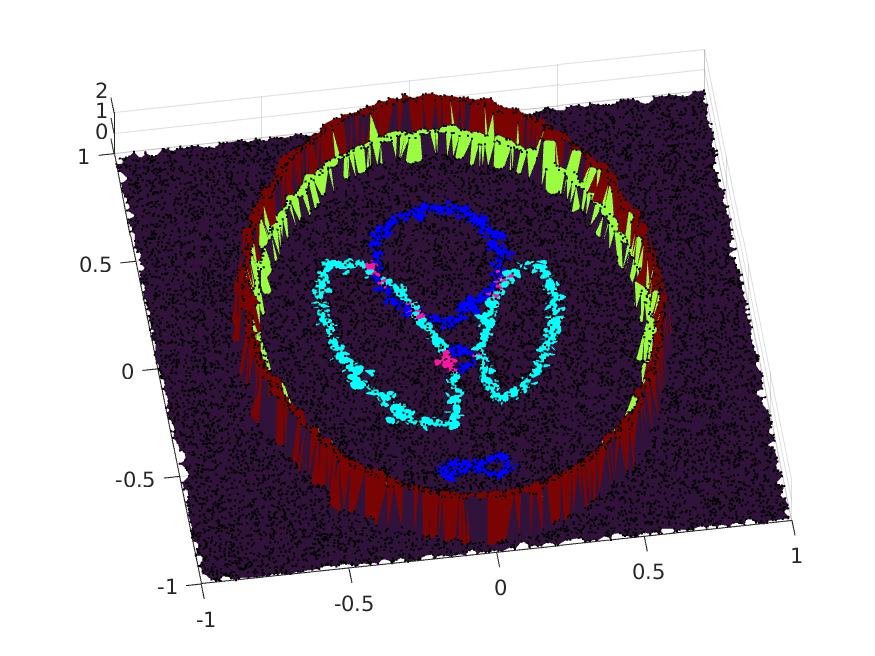}

}

\caption{Original functions (a) $f_{3}$ and (b) $f_{4}$, and their approximated
local jump functions (c) and (d), respectively.\label{fig:Original-function-and}}
\end{figure}

A key feature of RDI, which is also partially the reason for its larger
stencils, is that it can detect both $C^{0}$ and $C^{1}$ discontinuities.
In \citep{saxena2009high}, MED was extended to detect $C_{1}$ discontinuities
by first estimating the derivatives. To compare with \citep{saxena2009high},
we use the test function $f_{5}$ from \citep{saxena2009high},
\[
f_{5}=\begin{cases}
-(\sqrt{x^{2}+y^{2}}-\frac{1}{2})+\frac{1}{12}\sin(2\pi\sqrt{x^{2}+y^{2}}), & \textrm{if }\sqrt{x^{2}+y^{2}}<\frac{1}{2},\\
(\sqrt{x^{2}+y^{2}}-\frac{1}{2})+\frac{1}{12}\sin(2\pi\sqrt{x^{2}+y^{2}}), & \textrm{if }\sqrt{x^{2}+y^{2}}\geq\frac{1}{2},
\end{cases}
\]
which has a $C_{1}$ discontinuity at the origin and along the unit
circle. We generated three different meshes over $[-1,1]\times[-1,1]$
of similar sizes as those in \citep{saxena2009high}, including one
uniform mesh and two irregular meshes. As shown in Figure~\ref{fig:Discontinuities-of-f5},
RDI detected all the $C^{1}$ discontinuities similar to MED, except
for the discontinuities at the center over the uniform mesh. It is
worth noting that RDI produced fewer false positives than those in
\citep{saxena2009high}, probably because the WLS scheme used in RDI
is less prone to noise. In addition, RDI is potentially more efficient
than MED, since the most expensive part of RDI, namely the OSUS operator,
is independent of the function values and hence can be constructed
in an offline preprocessing step. In contrast, the operators in MED
depend on function values and cannot be evaluated offline.

\begin{figure}
\subfloat[$f_{5}$]{\includegraphics[width=0.48\columnwidth]{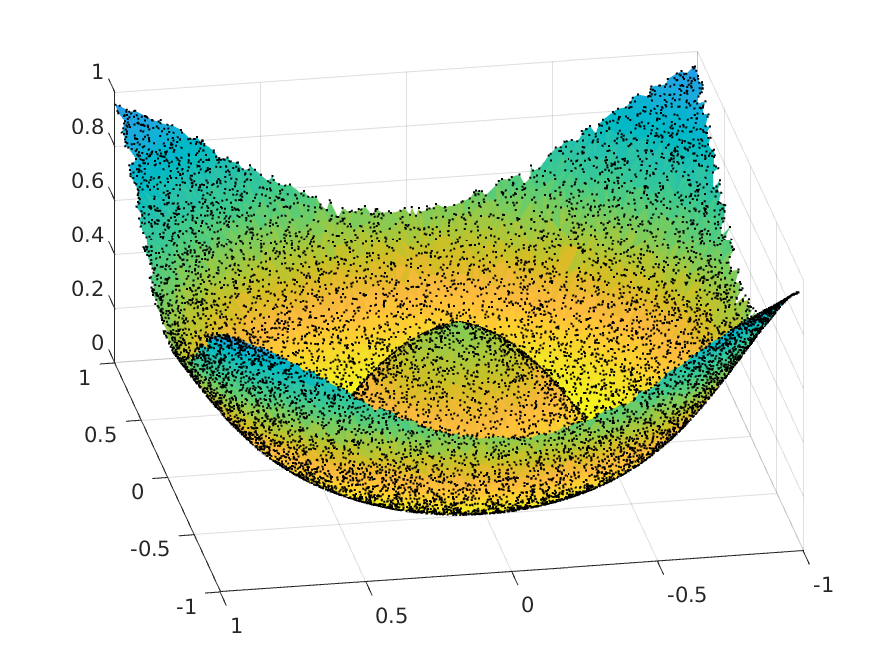}

}\hfill\subfloat[On a uniform grid of size $16,384$.]{\includegraphics[width=0.48\columnwidth]{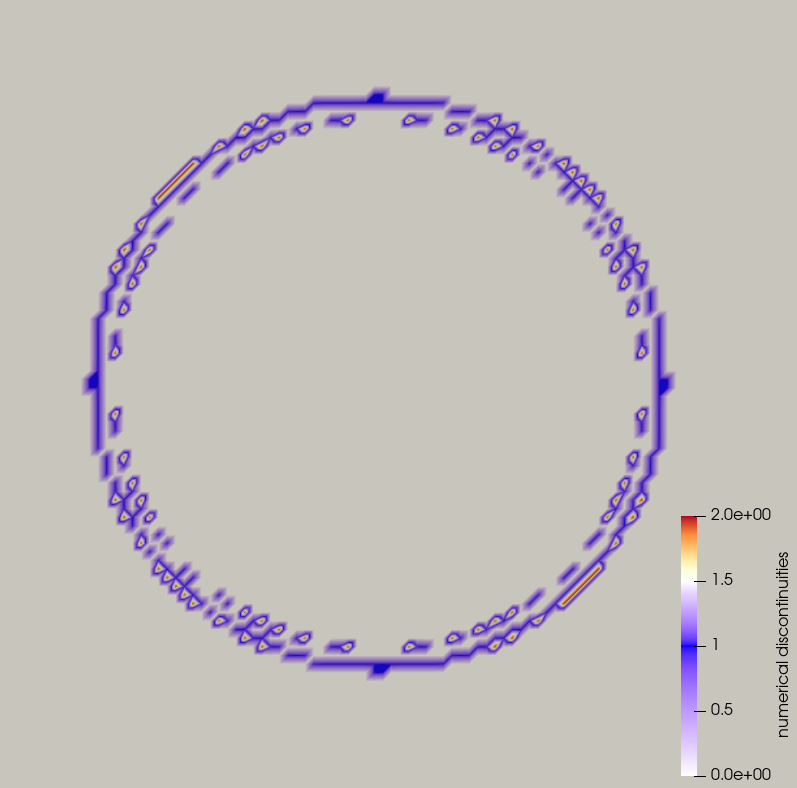}

}

\subfloat[On the Delaunay triangulation of $16,384$ randomly sampled points.]{\includegraphics[width=0.48\columnwidth]{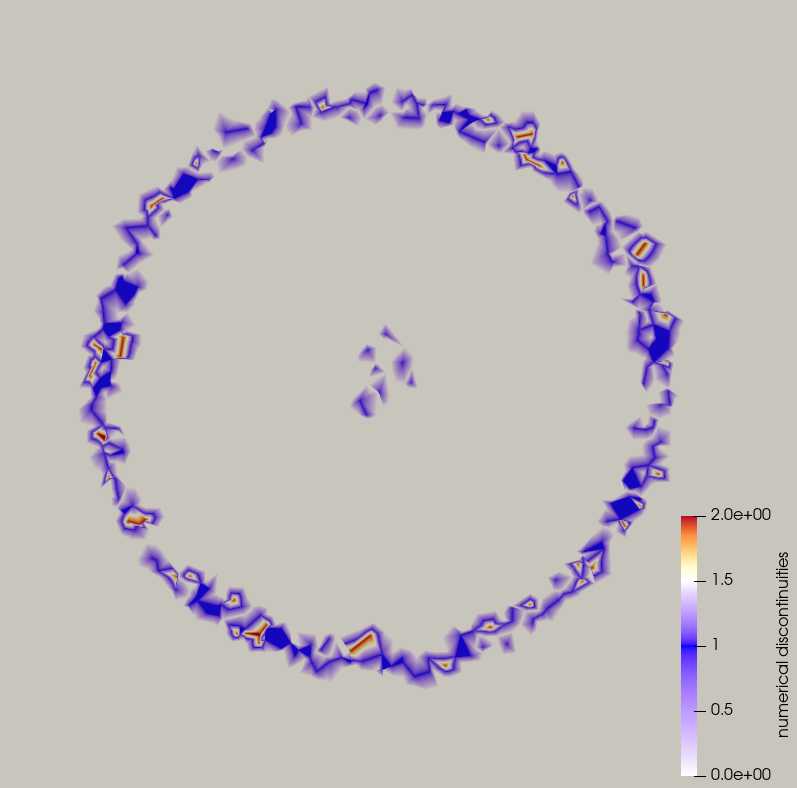}

}\hfill\subfloat[On the Delaunay triangulation of $65,536$ randomly sampled points.]{\includegraphics[width=0.48\columnwidth]{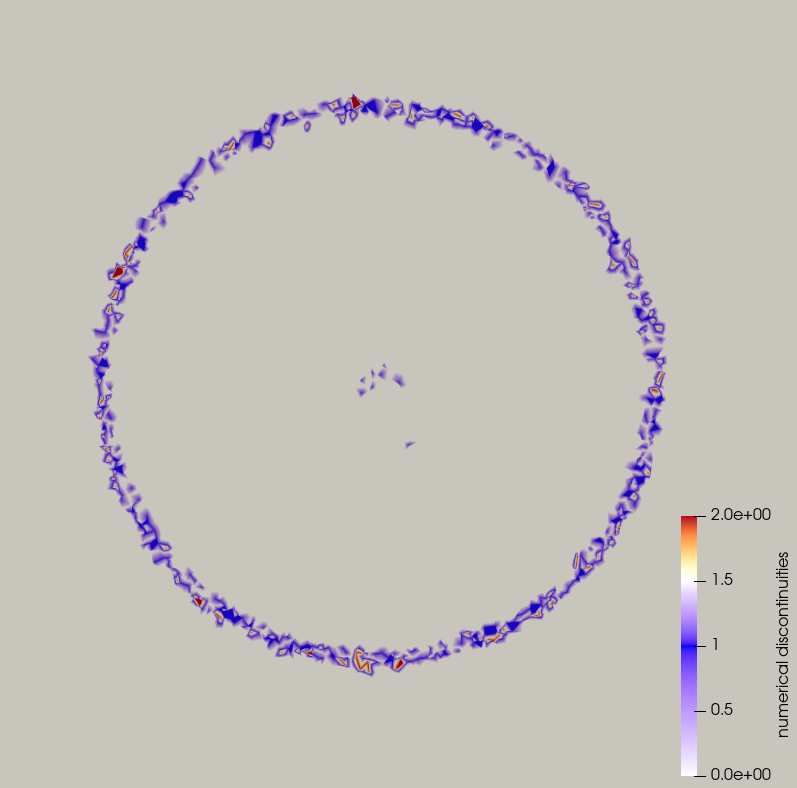}

}

\caption{Function $f_{5}$ and $C_{1}$ discontinuities detected by RDI on
three different meshes.\label{fig:Discontinuities-of-f5}}
\end{figure}

\subsection{Generalization to surfaces with sharp features\label{subsec:Numerical-results-with-sharp-features}}

Our previous examples employed simple smooth geometries, such as spheres
and planes. We now consider surfaces with sharp features, which present
additional complications. Similar to \citep{li2019compact}, we virtually
split the surface mesh into smooth patches along the feature curves
and then employ one-sided stencils along the sharp features. The virtual
splitting help prevent many false positives and negatives, as we demonstrate
in \ref{sec:virtual split and geometric discontinuities}. To maintain
stability, we increase the stencil sizes near the patch boundaries
when computing the OSUS operator, as discussed in Section~\ref{subsec:Generalization-to-surfaces}.

We carried out several experiments and displayed the results of scaled
functions $f_{6}$ and $f_{7}$, which are defined as
\[
f_{6}(x,y,z)=\tanh(x)\textrm{sign}(y)+\tanh(y)\textrm{sign}(z)+\tanh(z)\textrm{sign}(x),
\]
\[
f_{7}(x,y,z)=(\max\{x\}-\min\{x\})g\left(\frac{x-\min\{x\}}{\max\{x\}-\min\{x\}}\right),
\]
respectively, where 
\[
g(x)=\begin{cases}
x, & \textrm{if }0\leq x<\frac{1}{4},\\
\frac{1}{2}-x, & \textrm{if }\frac{1}{4}\leq x<\frac{1}{2},\\
\frac{3}{4}, & \textrm{if }\frac{1}{2}\leq x<\frac{3}{4},\\
16(x-\frac{3}{4})^{3}+\frac{3}{4}, & \textrm{if }\frac{3}{4}\leq x\leq1.
\end{cases}
\]

\begin{figure}
\subfloat[$f_{6}$]{\includegraphics[width=0.48\columnwidth]{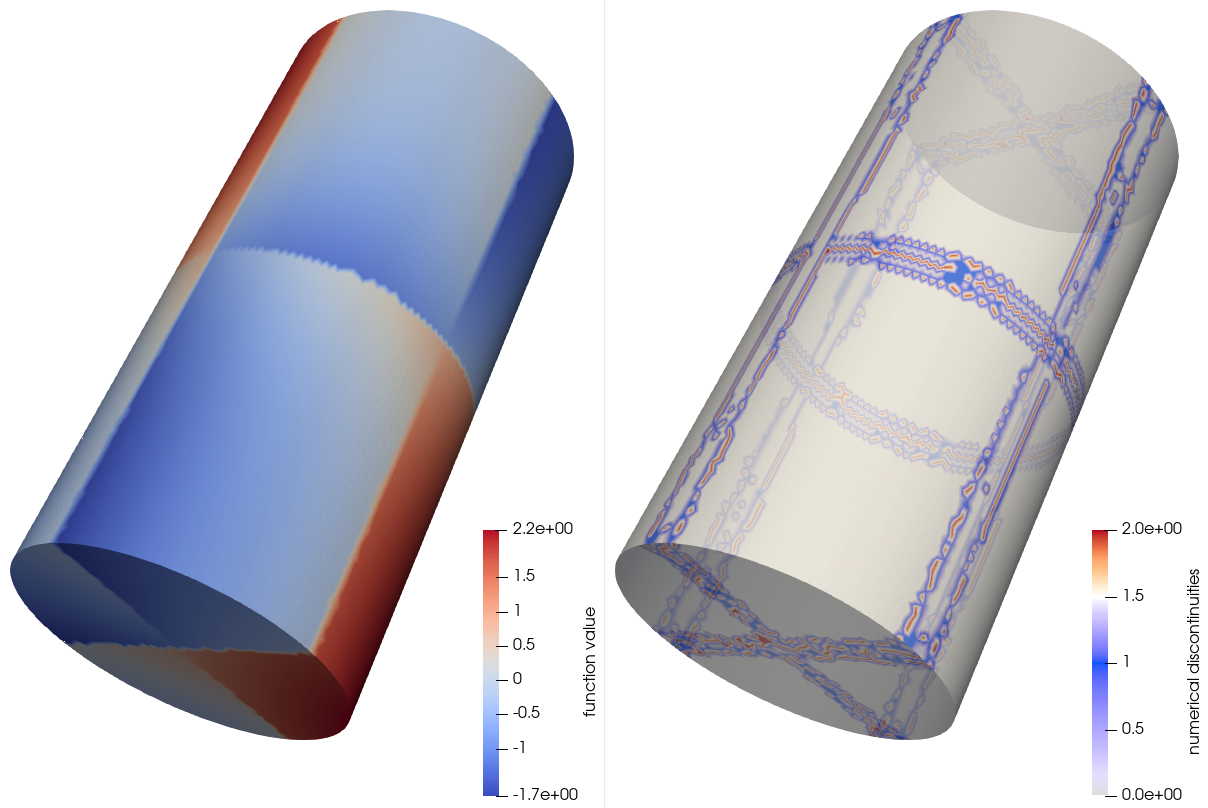}

}\subfloat[$f_{7}$]{\includegraphics[width=0.48\columnwidth]{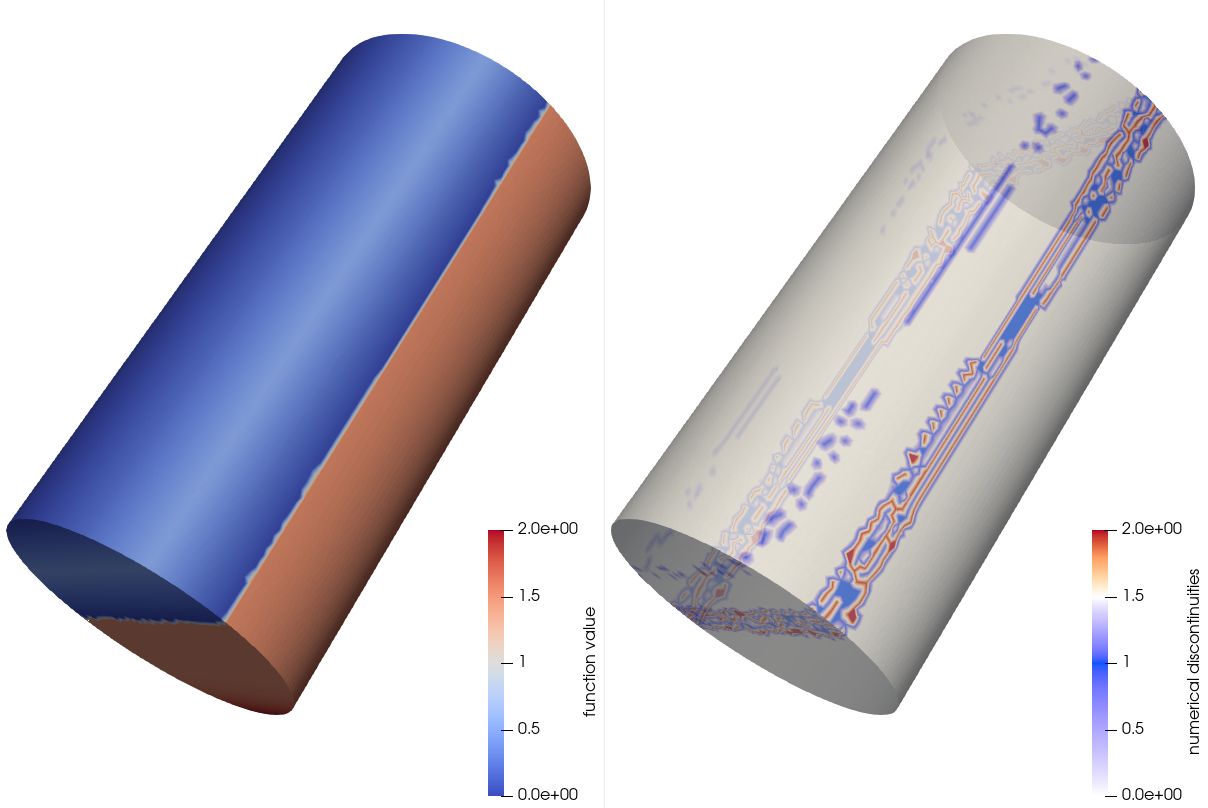}

}

\caption{Function value and numerical discontinuities detected by RDI without
virtual splitting on the surface triangulations of a cylinder.\label{fig:on cylinder}}
\end{figure}

\begin{figure}
\subfloat[$f_{6}$]{\includegraphics[width=0.48\columnwidth]{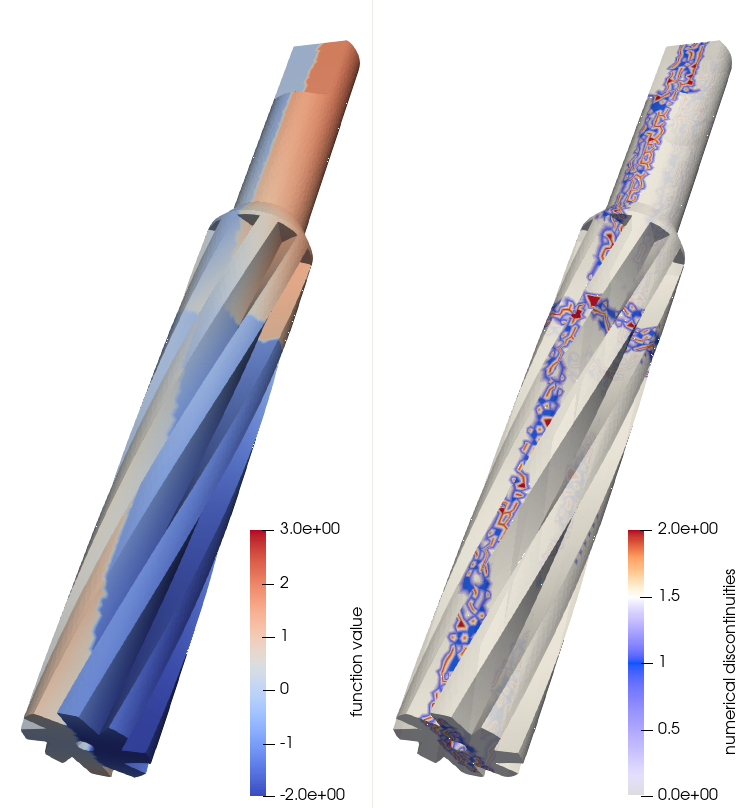}

}\subfloat[$f_{7}$]{\includegraphics[width=0.48\columnwidth]{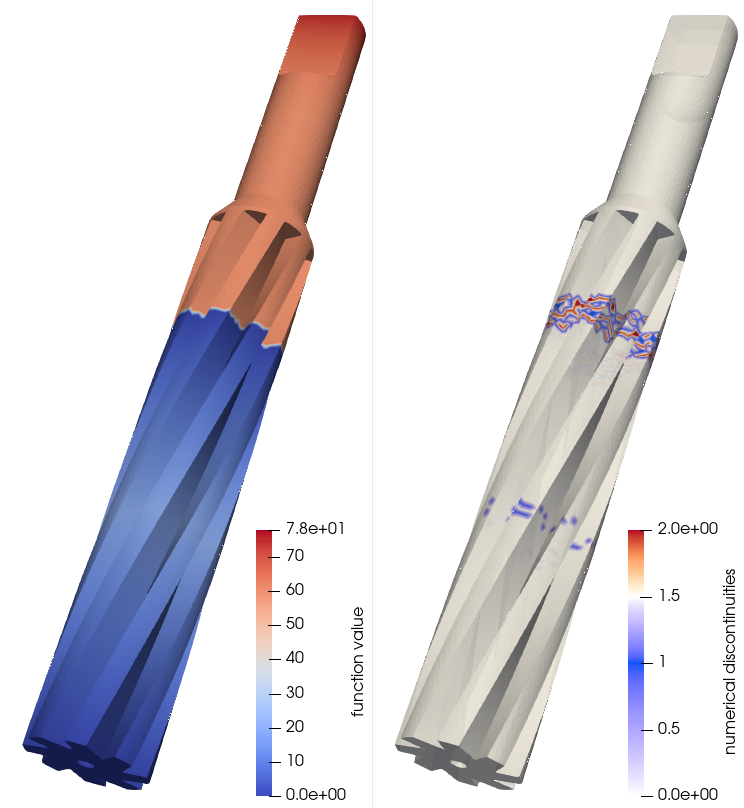}

}

\caption{Function value and numerical discontinuities detected by RDI with
virtual splitting on the surface triangulation of a reamer.\label{fig:on reamer}}
\end{figure}

Figures~\ref{fig:on cylinder} and \ref{fig:on reamer} display the
detected discontinuities on a cylinder and a reamer, respectively.
It can be observed that RDI accurately detected the $C^{0}$ discontinuities.
However, it missed part of the $C^{1}$ discontinuities as the quadratic
WLS had minimal overshoots and undershoots for the given mesh resolution
in the corresponding regions. Decreasing the threshold $\kappa$ would
make RDI more sensitive and miss fewer $C^{1}$ discontinuities, although
it may also introduce some false positives.

\subsection{Application to remap\label{subsec:Application-to-remap}}

Robust detection of discontinuities has a variety of applications.
In this section, we demonstrate the use of RDI in remapping data between
different meshes, also known as data remap. In this context, upon
identifying $C^{0}$ and $C^{1}$ discontinuities, one can employ
numerical techniques, such as WLS-ENO remap \citep{li2019compact}
and CAAS \citep{bradley2019communication}, to resolve the Gibbs phenomena
\citep{gottlieb1997gibbs}. Figure~\ref{fig:Comparison-between-WLS-WLS-ENO}
compares the remapping of $f_{2}$ using WLS and WLS-ENO on a sphere,
where the former does not distinguish smooth and discontinuous regions
and the latter adjusts the weighting schemes at discontinuities. As
can be seen, the WLS result suffers from oscillations near discontinuities.
In contrast, the combination of WLS-ENO and RDI successfully eliminated
overshoots and undershoots while preserving accuracy in smooth regions.

\begin{figure}
\subfloat[$f_{2}$ on coarse SCVT with intersection plane $y=0$.]{\includegraphics[width=0.35\columnwidth]{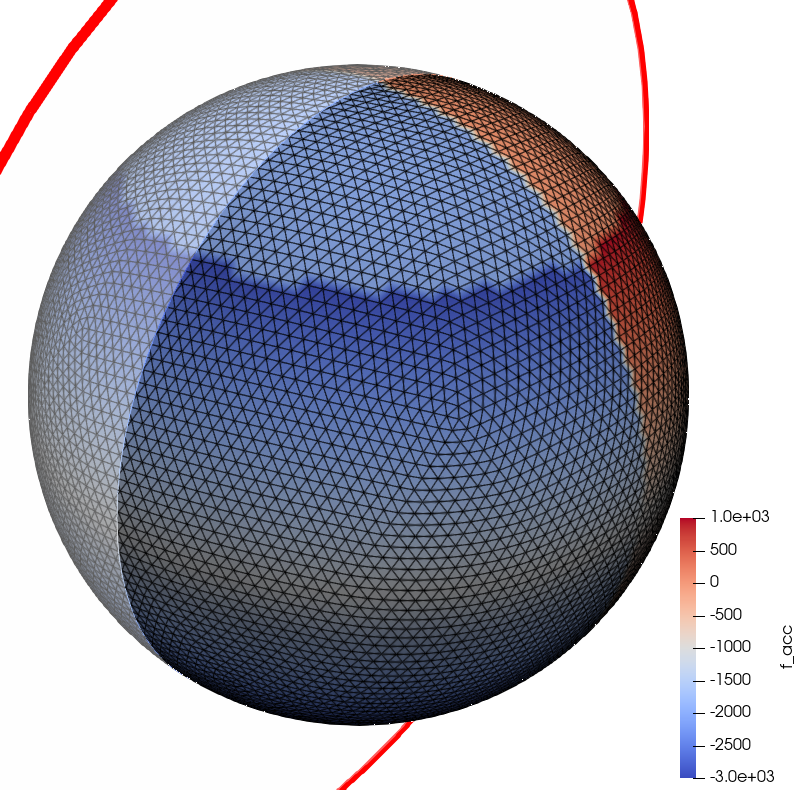}

}$\quad$\subfloat[Exact and remapped function values on cross-section curve.]{\includegraphics[width=0.61\columnwidth]{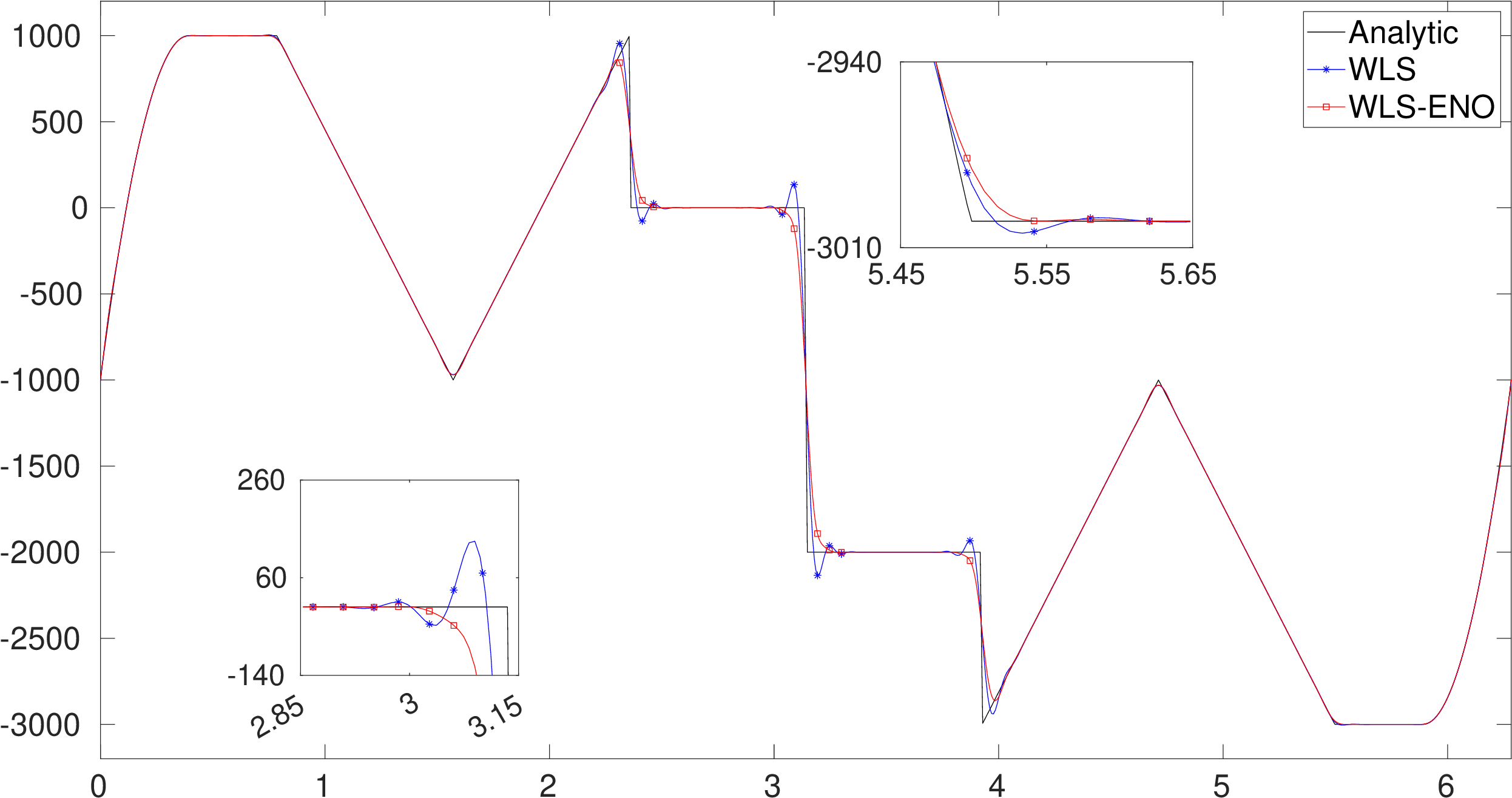}

}

\caption{Comparison between original WLS and WLS-ENO without and with detection
and treatment of discontinuities, respectively.\label{fig:Comparison-between-WLS-WLS-ENO}}
\end{figure}

\section{\label{sec:Conclusions}Conclusions}

In this paper, we have introduced the Robust Discontinuity Indicators
(RDI) method, an innovative approach for the robust and efficient
detection of discontinuities in the approximation of piecewise continuous
functions over meshes. Our experiments and analyses have demonstrated
the potential of the RDI in identifying both $C^{0}$ and $C^{1}$
discontinuities in node-based values and in addressing the challenges
associated with non-uniform meshes and complex surface geometries.

Despite the promising results obtained, several avenues exist for
the further improvement and expansion of the RDI method. One potential
enhancement involves incorporating a thinning process to refine the
detected discontinuities further, thereby making them more precise.
The implementation of this thinning process could significantly improve
the fidelity of the approximation and further reduce the occurrence
of the Gibbs phenomena, thereby enhancing the overall quality of the
computational solution. Moreover, the present work primarily focused
on 2D surfaces, creating an opportunity for the extension of the RDI
method to three dimensions. Such an extension would significantly
broaden the application range of the RDI method, allowing it to accommodate
more complex real-world scenarios found in fields such as fluid dynamics,
material science, and medical imaging.

In conclusion, the RDI method represents a significant step forward
in the robust and efficient detection of discontinuities in piecewise
continuous function approximations. The potential improvements and
expansions identified will pave the way for further research and development
in this critical area, with the ultimate goal of achieving more accurate
and efficient computational methods.

\section*{Acknowledgments}

This work was supported under the Scientific Discovery through Advanced
Computing (SciDAC) program in the US Department of Energy\textquoteright s
Office of Science, Office of Advanced Scientific Computing Research
through subcontract \#462974 with Los Alamos National Laboratory.

\bibliographystyle{elsarticle-num}
\bibliography{solution_transfer}

\appendix

\section{The Effect of Virtual Splitting along Sharp Features\label{sec:virtual split and geometric discontinuities}}

A general surface might have geometric discontinuities, such as features
and ridge points. Running RDI on such geometric discontinuities directly
might yield inaccurate detection of numerical discontinuities. Specifically,
a smooth function in the Euclidean space could often be misclassified
as $C^{1}$ discontinuity near the geometric feature and ridge points.
Conversely, a function with $C^{1}$ discontinuity that coincide with
the geometric discontinuity might be mistaken as a smooth function. 

To illustrate these two cases, let us first consider two functions
$f_{6}(x,y,z)=g_{1}(x,y,z,0.5)$ and $f_{7}(x,y,z)=g_{1}(x,y,z,1)$
on a curve $\Gamma=\{(t,g_{2}(t),t)|t\in\mathbb{R}\}$, where 
\[
g_{1}(x,y,z,a)=\begin{cases}
a+\frac{3}{2}(y-a), & z>a,\\
x, & z\leq a,
\end{cases}
\]
and 
\[
g_{2}(t)=\begin{cases}
1+\frac{1}{2}(t-1), & t>1,\\
t, & t\leq1,
\end{cases}
\]
are two piecewise smooth functions. $g_{1}(x,y,z,a)$ has a $C^{1}$
discontinuity at $z=a$ on $\varGamma$ if $a\leq1$, since the gradients
are $(0,\frac{3}{2},0)$ and $(0,0,1)$ on two sides of $z=a$. However,
using RDI directly without virtual splitting would result in false
positives for $f_{6}$ and false negatives for $f_{7}$, respectively.

\begin{figure}
\subfloat[$\alpha$]{\includegraphics[width=0.48\columnwidth]{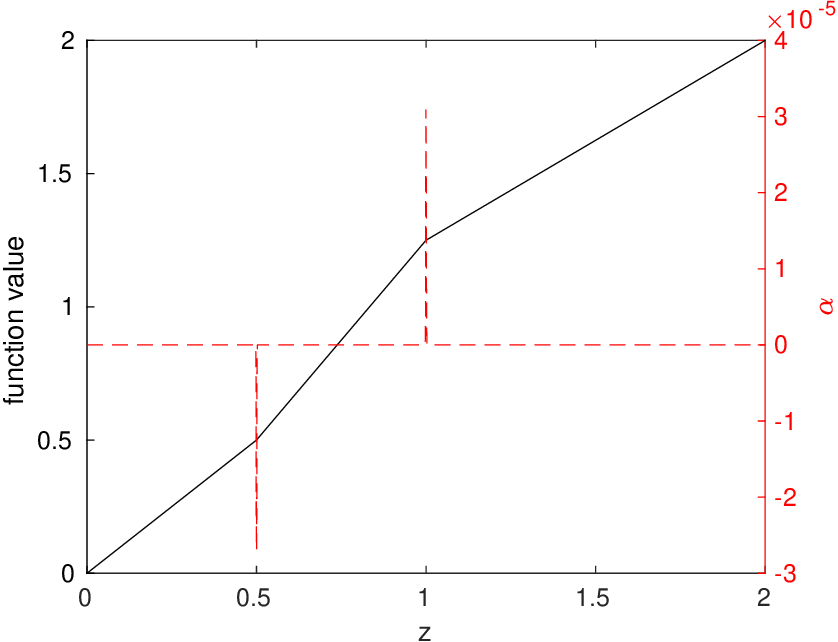}

}$\quad$\subfloat[$\beta$]{\includegraphics[width=0.48\columnwidth]{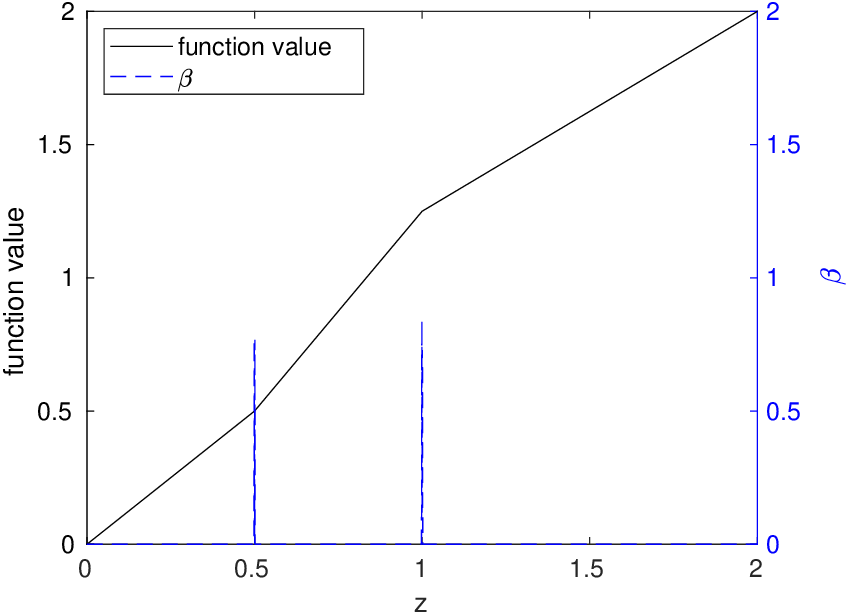}

}

\caption{$\alpha$ and $\beta$ values of $f_{6}$ on $\varGamma$. The black
line is the function value of $f_{6}$. In (a), the red dash line
is $\alpha$. In (b), the blue dash line is $\beta$. \label{fig:a-and-b-f6}}
\end{figure}

Figure~\ref{fig:a-and-b-f6} shows that RDI without virtual splitting
can detect the $C^{1}$ discontinuity at $z=a=0.5$. However, it also
erroneously classified the smooth region near $z=1$ as $C^{1}$ discontinuity.
Figure~\ref{fig:a-and-b-f7} illustrates that RDI without virtual
splitting could miss the $C^{1}$ discontinuity at $z=a=1$ because
$\beta\equiv0$.

\begin{figure}
\subfloat[$\alpha$]{\includegraphics[width=0.48\columnwidth]{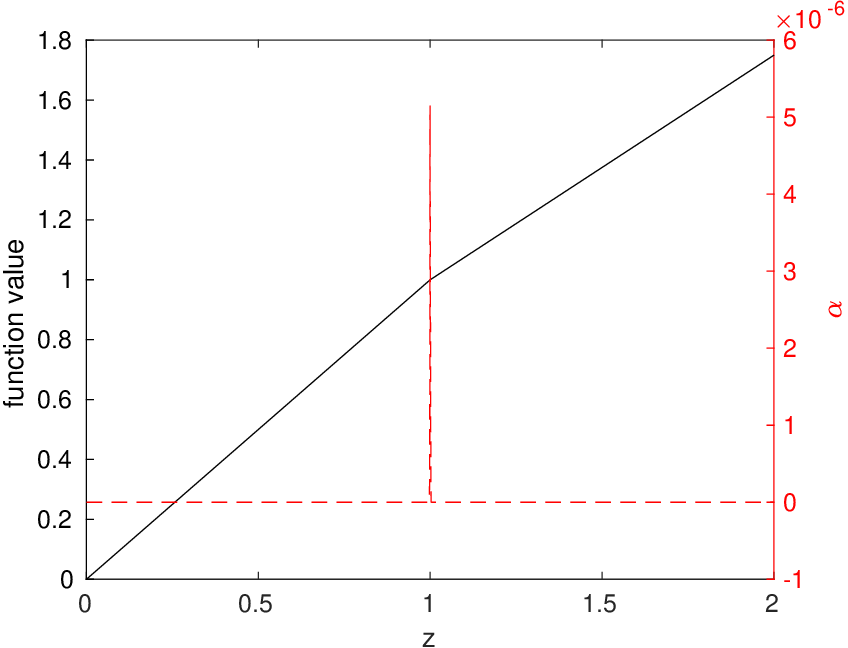}

}$\quad$\subfloat[$\beta$]{\includegraphics[width=0.48\columnwidth]{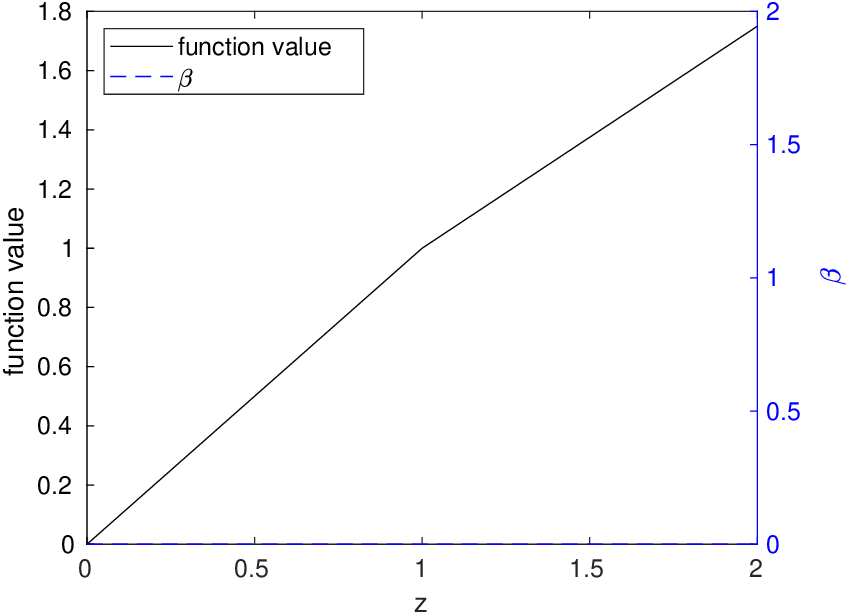}

}

\caption{$\alpha$ and $\beta$ of $f_{7}$ on $\varGamma$. The black line
is the function value of $f_{7}$. In (a), the red dash line is $\alpha$.
In (b), the blue dash line is $\beta$. \label{fig:a-and-b-f7}}
\end{figure}

 Similar misclassifications can also occur on surface meshes without
virtual splitting. For instance, Figure~\ref{fig:cylinder-f6} illustrates
that RDI without virtual splitting on the triangulation of a cylinder
for a simple smooth function $f_{6}(x,y,z)=z+2$. The smooth function
$f_{6}$ in the Euclidean space is constant on the two bases and linear
on the lateral surface. RDI with virtual splitting produced expected
results. In contrast, RDI without virtual splitting would yield false
positives near the features on the bases. 

\begin{figure}
\includegraphics[width=0.98\columnwidth]{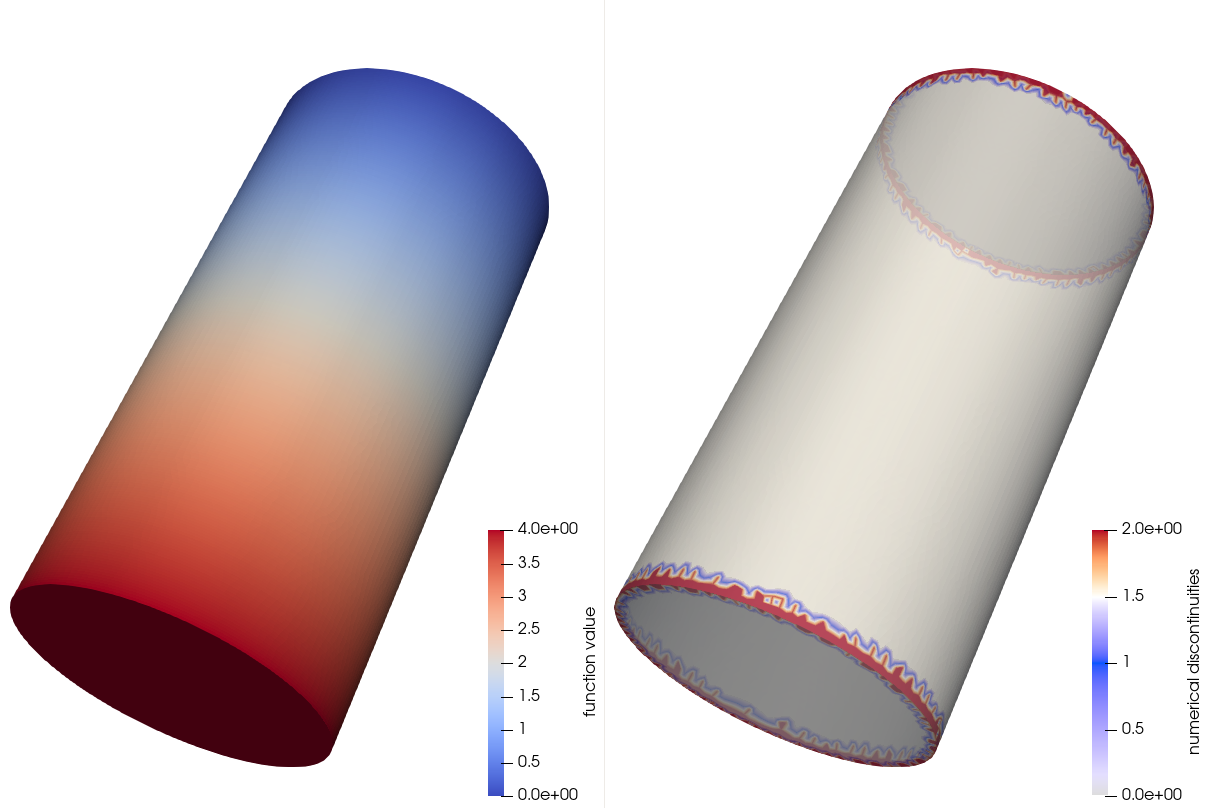}

\caption{Function value (left) and numerical discontinuities detected by RDI
without virtual splitting (right) of $f_{6}$ on the triangulation
of a cylinder.\label{fig:cylinder-f6}}

\end{figure}

\end{document}